\documentclass[12pt]{article}
\usepackage{graphicx}
\usepackage{amsmath,amsthm,amssymb,enumerate}
\usepackage{euscript,mathrsfs}
\usepackage[left=2cm,right=2cm,top=3.5cm,bottom=3.5cm]{geometry}
\usepackage{color}
\catcode`\@=11 \@addtoreset{equation}{section}

\catcode`\@=12

\usepackage{amsfonts}
\usepackage{color}
\usepackage{framed}
\usepackage{bm}
\usepackage{mathtools}
\usepackage{cite}
\usepackage{pgf}
\usepackage{lmodern}
\usepackage{pifont}
\usepackage{framed}
\usepackage{enumitem}
\usepackage{stmaryrd}
\usepackage{amssymb}

\allowdisplaybreaks

\newtheorem{Theorem}{Theorem}[section]
\newtheorem{Proposition}[Theorem]{Proposition}
\newtheorem{Lemma}[Theorem]{Lemma}
\newtheorem{Corollary}[Theorem]{Corollary}

\theoremstyle{definition}
\newtheorem{Definition}[Theorem]{Definition}

\newtheorem{Remark}[Theorem]{Remark}

\newcommand{\bTheorem}[1]{
\begin{Theorem} \label{T#1} }
\newcommand{\eT}{\end{Theorem}}

\newcommand{\bProposition}[1]{
\begin{Proposition} \label{P#1}}
\newcommand{\eP}{\end{Proposition}}

\newcommand{\bLemma}[1]{
\begin{Lemma} \label{L#1} }
\newcommand{\eL}{\end{Lemma}}

\newcommand{\bCorollary}[1]{
\begin{Corollary} \label{C#1} }
\newcommand{\eC}{\end{Corollary}}

\newcommand{\bRemark}[1]{
\begin{Remark} \label{R#1} }
\newcommand{\eR}{\end{Remark}}

\newcommand{\bDefinition}[1]{
\begin{Definition} \label{D#1} }
\newcommand{\eD}{\end{Definition}}

\newcommand{\vw}{\vc{w}}

\newcommand{\bfq}{\mathbf{q}}
\newcommand{\bfg}{\mathbf{g}}

\newcommand{\bfe}{\mathbf{e}}
\newcommand{\bfn}{\mathbf{n}}

\newcommand{\bfF}{\mathbf{F}}

\newcommand{\bfphi}{\boldsymbol{\varphi}}

\newcommand{\bFormula}[1]{
\begin{equation} \label{#1}}
\newcommand{\eF}{\end{equation}}

\newcommand{\Ov}[1]{\overline{#1}}

\newcommand{\aleq}{\stackrel{<}{\sim}}

\newcommand{\Un}[1]{\underline{#1}}
\newcommand{\vr}{\varrho}

\newcommand{\tvr}{\tilde \vr}
\newcommand{\tvu}{{\tilde \vu}}
\newcommand{\tvt}{\tilde \vt}

\newcommand{\vt}{\vartheta}
\newcommand{\vu}{\vc{u}}
\newcommand{\vm}{\vc{m}}

\newcommand{\vc}[1]{{\bf #1}}

\newcommand{\Div}{{\rm div}_x}
\newcommand{\Grad}{\nabla_x}
\newcommand{\Dt}{\frac{\rm d}{{\rm d}t}}

\newcommand{\dx}{\,{\rm d} {x}}

\newcommand{\dt}{\,{\rm d} t }

\newcommand{\vU}{\vc{U}}
\newcommand{\vV}{\vc{V}}

\newcommand{\intO}[1]{\int_{\Omega} #1 \ \dx}
\newcommand{\intOh}[1]{\int_{\Omega_h} #1 \ \dx}

\newcommand{\cg}{\color{black}}

\def\softd{{\leavevmode\setbox1=\hbox{d}%
          \hbox to 1.05\wd1{d\kern-0.4ex{\char039}\hss}}}

\newcommand{\defeq}{\vcentcolon=}

\newcommand{\di}{\textnormal{div}_x\,}
\newcommand{\dih}[1]{\left(\textnormal{div}_h\,#1\right)_K}
\newcommand{\diht}[1]{\left(\widetilde{\textnormal{div}_h}\,#1\right)_K}
\newcommand{\laph}[1]{\left(\Delta_h\,#1\right)_K}
\newcommand{\laphs}[1]{\left(\Delta^s_h\,#1\right)_K}

\newcommand{\pd}[2]{\partial_{#2}\, #1}
\newcommand{\DDt}[3]{\frac{\textnormal{d}^{#1}{#2}}{\textnormal{d}{#3}^{#1}}}

\newcommand{\grad}{\nabla_x\, }

\newcommand{\bfQ}{\mathbf{Q}}
\newcommand{\bff}{\mathbf{f}}
\newcommand{\bfG}{\mathbf{G}}

\newcommand{\cv}{1/(\gamma-1)}

\newcommand{\avrgK}[1]{(\widetilde{#1})^s_{K}}
\newcommand{\avrg}[1]{(\Ov{#1})_{\sigma}}

\newcommand{\jump}[1]{\llbracket #1 \rrbracket_{\sigma}}

\newcommand{\n}[2]{\left\lVert{#2}\right\rVert_{#1}}

\newcommand{\grid}{\mathcal{T}_h}
\newcommand{\elem}{K}
\newcommand{\elemL}{L}

\newcommand{\edge}{\mathcal{E}}
\newcommand{\bndr}{\sigma,s\pm}

\newcommand{\bv}[1]{{#1}_{\sigma}}

\newcommand{\dv}[1]{{#1}_{K}}
\newcommand{\dvpp}[1]{{#1}_{L}}

\newcommand{\dvmm}[1]{{#1}_{J}}

\newcommand{\dvxp}[1]{{#1}_{i+\frac{1}{2},j}}
\newcommand{\dvxm}[1]{{#1}_{i-\frac{1}{2},j}}

\newcommand{\dvyp}[1]{{#1}_{i,j+\frac{1}{2}}}
\newcommand{\dvym}[1]{{#1}_{i,j-\frac{1}{2}}}

\newcommand{\fluxder}[1]{\left(\partial_h^s {#1}\right)_K}
\newcommand{\der}[1]{\left(\widetilde{\partial_h^s} {#1}\right)_{K}}
\newcommand{\derp}[1]{\left(\partial_h^{s+} {#1}\right)_{K}}
\newcommand{\derm}[1]{\left(\partial_h^{s-} {#1}\right)_{K}}

\newcommand{\derb}[1]{\left(\partial_h^{s+} {#1}\right)_{\sigma}}

\newcommand{\intOT}[1]{\int_{0}^T {#1}\dt }

\newcommand{\intOhi}[1]{\int_{K}{#1}\dx}

\definecolor{Cgrey}{rgb}{0.85,0.85,0.85}
\definecolor{Cblue}{rgb}{0.50,0.85,0.85}
\definecolor{Cred}{rgb}{1,0,0}
\definecolor{fancy}{rgb}{0.10,0.85,0.10}

\newcommand\Cbox[2]{%
    \newbox\contentbox%
    \newbox\bkgdbox%
    \setbox\contentbox\hbox to \hsize{%
        \vtop{
            \kern\columnsep
            \hbox to \hsize{%
                \kern\columnsep%
                \advance\hsize by -2\columnsep%
                \setlength{\textwidth}{\hsize}%
                \vbox{
                    \parskip=\baselineskip
                    \parindent=0bp
                    #2
                }%
                \kern\columnsep%
            }%
            \kern\columnsep%
        }%
    }%
    \setbox\bkgdbox\vbox{
        \color{#1}
        \hrule width  \wd\contentbox %
               height \ht\contentbox %
               depth  \dp\contentbox
        \color{black}
    }%
    \wd\bkgdbox=0bp%
    \vbox{\hbox to \hsize{\box\bkgdbox\box\contentbox}}%
    \vskip\baselineskip%
}

\begin{document}


\title{Convergence of finite volume schemes for the Euler equations via dissipative
measure--valued solutions}

\author{Eduard Feireisl
\thanks{The research of E.F. and H.M.~leading to these results has received funding from the
European Research Council under the European Union's Seventh
Framework Programme (FP7/2007-2013)/ ERC Grant Agreement
320078. The Institute of Mathematics of the Academy of Sciences of
the Czech Republic is supported by RVO:67985840.} \and M\' aria Luk\' a\v cov\' a-Medvi\softd ov\'a
\thanks{The research of M.L. was supported by the German Science Foundation under the Collaborative Research Centers TRR~146 and TRR~165.}
 \and Hana Mizerov\' a \footnotemark[1]
}

\date{\today}

\maketitle

\bigskip

\centerline{$^*$ Institute of Mathematics of the Academy of Sciences of the Czech Republic}

\centerline{\v Zitn\' a 25, CZ-115 67 Praha 1, Czech Republic}

\centerline{feireisl@math.cas.cz}
\centerline{mizerova@math.cas.cz}

\bigskip

\medskip
\centerline{$^\dagger$ Institute of Mathematics, Johannes Gutenberg-University Mainz}

\centerline{Staudingerweg 9, 551 28 Mainz, Germany}

\centerline{lukacova@uni-mainz.de}

\begin{abstract}
The Cauchy problem for the complete Euler system is in general ill posed in the class of admissible (entropy producing)
weak solutions. This suggests there might be sequences of approximate solutions that develop fine scale oscillations. Accordingly, the concept of \emph{measure--valued solution} that capture possible oscillations is more suitable for analysis.
We study the convergence of a class of entropy stable finite volume schemes for the barotropic  and complete compressible Euler equations
in the multidimensional case. We establish suitable stability and consistency estimates and show that the
Young measure generated by numerical solutions represents a dissipative measure--valued
solution of the Euler system. Here dissipative means that a suitable form of the Second law of thermodynamics is incorporated in the definition of the measure--valued solutions. In particular, using the recently
established weak-strong uniqueness principle,  we
show that the numerical solutions converge pointwise to the regular solution of the limit systems
at least on the lifespan of the latter.
\end{abstract}



\tableofcontents

\section{Introduction}

The Euler equations of compressible fluid flow
represent the simplest possible model that incorporates all fundamental principles of thermodynamics including the Second law
usually expressed in terms of the entropy balance appended as an \emph{admissibility condition} to the system. The entropy
should be produced by any physically realistic process and this criterion is supposed to rule out the unphysical solutions
that may still satisfy the basic system in the sense of distributions. In addition, the entropy balance provides
crucial \emph{a priori} bounds, in particular, positivity of the pressure when the system is written in the so--called conservative
variables.

Another characteristic feature of the Euler system
is that discontinuities may develop after a finite time even if the initial data are smooth. It is therefore quite natural to look for a weaker representation of solutions,
for instance the weak solutions that satisfy the underlying equations in the sense of distributions, see
\cite{Dafermos00, GR96, K97, LV02, S99} and the references therein.
It is also a well-known fact that such weak solutions may fail to be unique, and, consequently, the Second law of thermodynamics has been
proposed as a selection criterion. Although the entropy production principle has been efficient
in the case of scalar multidimensional hyperbolic conservation laws as well as the one-dimensional systems, see
\cite{kr, glimm, bressan1, bressan}, it completely fails in the multidimensional setting.
Recently, it has been shown by De Lellis and Sz\'ekelyhidi \cite{dlsz1, dlsz2} and Chiodaroli et al.~\cite{kreml} that infinitely many weak entropy solutions can be constructed for the multidimensional barotropic Euler equations. These  so-called wild solutions seem to behave unphysically as they may produce energy. These results has been extended in  al.~\cite{FKKM17} to the complete multidimensional Euler system in the class of $L^\infty$ weak admissible solutions. In particular, these solutions satisfy the energy balance together with the entropy inequality; whence they are compatible with both the First and the Second law of thermodynamics.

Inspired by the previous results as well as by the numerical analysis performed in \cite{FMT16},
we examine stability and convergence of certain numerical schemes in the class of so-called
{\em dissipative measure--valued solutions}, see Section~2 and 3. The concept of measure-valued solutions for conservation laws is not new, see, e.g., \cite{Diperna83, Diperna85, SW12, BF, FKKM17, KronZaja} and the references therein.
However, the recently introduced class of dissipative measure-valued solutions is particularly suitable since the weak-strong uniqueness holds and the dissipative measure-valued solution coincides with the classical solution as far as the latter exists \cite{GSWW, BF, BreFei17B, FKKM17}. Similar concept has been adopted by Tzavaras et al.
\cite{DeStTz}, \cite{CrGaTz},
in the context of elastodynamics, thermoelasticity, and other related
problems.

As is well known, the entropy stability of a numerical scheme plays a crucial role
in the convergence analysis of numerical solutions. Construction of entropy conservative schemes has been introduced by Tadmor in a seminal paper \cite{Tad87}. This concept has been later used to study entropy stability of numerical schemes, we refer the reader to
\cite{ABBSM16, BB02, B03, CL96, FMT16, FKMT17, LMR02, Tad03} and the references therein.

There is a considerable body of literature dealing with the convergence of numerical schemes for multidimensional hyperbolic conservation laws. Though the chosen techniques depend on the assumptions imposed on exact solutions,
a certain form of the discrete entropy inequality is indispensable.
Let us mention for example the results of Bouchut and Betherlin \cite{BB02, B03, FB05}, where the kinetic flux-splitting method has been used. Relying on the fully discrete entropy inequality and applying the method of DiPerna  \cite{Diperna83} and Tartar's results on compensated compactness they proved strong convergence of fully discrete kinetic flux-splitting scheme to the bounded weak entropy solution of isentropic Euler equations (or the shallow water equations \cite{BL17}) provided numerical solutions satisfies $L^\infty$-bounds and the vacuum does not appear.

In \cite{JR} Jovanovi\'{c} and Rohde assumed the existence of a classical solution to the Cauchy problem of a general multidimensional hyperbolic conservation law.
Applying the stability result for classical solutions in the class of entropy solutions due to Dafermos
\cite{Dafermos79} and DiPerna's method \cite{Diperna79, Diperna83}, they derived  error estimates for the explicit finite volume schemes satisfying the discrete entropy inequality and thus proved that the  numerical solutions convergence strongly to the exact classical solution.

In view of the fact that the classical solutions of hyperbolic conservation laws may not exist in general and in view of the recent results on non-uniqueness of weak entropy solutions \cite{dlsz1, dlsz2, kreml}, Fjordholm, Mishra and Tadmor revisited recently the question of convergence and proved that the semi-discrete entropy stable finite volume schemes converge to the measure--valued solutions provided numerical solutions satisfy $L^\infty$-bounds, coefficients of numerical viscosity are uniformly bounded from below by a positive constant  and the entropy Hessian is strictly positive definite, see \cite{fjdiss, FKMT17, FMT16}.

In contrast with the above works that are mostly devoted to general hyperbolic systems, we focus on the specific problems
in fluid mechanics represented through the complete Euler system, or its simplified barotropic analogue. Our framework are the
dissipative measure--valued solutions introduced in \cite{BF, BreFei17B, FKKM17}, see also the related numerical study for the isentropic Navier-Stokes equations~\cite{FL17}. In comparison with the previously used concept of measure--valued solutions, the existence of which is conditioned by mostly rather unrealistic assumptions of boundedness of certain physical quantities and the corresponding fluxes, the new framework accommodates the solutions generated by approximate sequences satisfying only the general
energy bounds.
Indeed, assuming only uniform lower  bound on the density and uniform upper bound on the energy
we show that the Lax-Friedrichs-type finite volume schemes generate the dissipative measure--valued solutions to the complete Euler equations.

The rest of the paper is organized as follows. In Section~2 we introduce the class of dissipative measure--valued (DMV) solutions to the barotropic and complete Euler systems and formulate the corresponding (DMV)-- strong uniqueness results. In Section~3 we recall a general concept of entropy stable finite volume schemes and introduce the local and global Lax-Friedrichs-type finite volume methods for the barotropic and complete Euler systems, respectively. Positivity of the pressure is studied in Section~4. Sections~5 and 6 are devoted to the stability and consistency of our numerical schemes. Finally, the limiting process is studied in Section~7. We will show that the numerical solutions generate a weakly-$(*)$ convergent subsequence and the Young measure that represents a (DMV) solution to the corresponding Euler system. Moreover, employing the (DMV)-- strong uniqueness principle, we will obtain strong (pointwise) convergence to the unique classical solution as long as the latter exists.

\section{Measure--valued solutions for the Euler system}
\label{m}

We consider the complete
Euler system describing the time evolution of a general compressible fluid and its isentropic (or more general barotropic) analogue that may be seen as the particular case when the entropy of the system is constant. We start with the simpler barotropic system. For the sake of simplicity, we will \emph{systematically} use the space--periodic boundary conditions
throughout the whole text. This means the underlying spatial domain can be identified with the flat torus
\begin{equation} \label{M2}
\Omega =\big( [0,1]|_{\{0,1\}}\big)^N,\ N=1,2,3.
\end{equation}
Note that, in this geometry, the physically more relevant impermeability condition
\[
{\bf u} \cdot \vc{n} |_{\partial \Omega} = 0
\]
can be accommodated in a direct fashion, see Ebin \cite{EB}.

\subsection{Measure--valued solutions for the barotropic Euler system}

Neglecting the influence of temperature fluctuations we can describe the motion of a compressible fluid by means of only
two basic state variables, the mass density $\vr = \vr(t,x)$ and the velocity field $\vu = \vu(t,x)$.
The resulting \emph{barotropic Euler system}
reads
\begin{equation} \label{M1}
\begin{split}
\partial_t \vr + \Div (\vr \vu) &= 0, \\
\partial_t (\vr \vu) + \Div (\vr \vu \otimes \vu) + \Grad p(\vr) &= 0,
\end{split}
\end{equation}
where $p = p(\vr)$ is the pressure. In what follows we focus on the isentropic pressure-density state equation
\begin{align}\label{pressure_bE}
p(\vr)=a\vr^{\gamma},\ \gamma >1.
\end{align}
Moreover, it is more convenient to study (\ref{M1}) in the \emph{conservative variables} $[\vr, \vc{m} = \vr \vu]$:
\begin{equation} \label{M1c}
\begin{split}
\partial_t \vr + \Div \vc{m} &= 0, \\
\partial_t \vc{m} + \Div \left( \frac{\vc{m} \otimes \vc{m} }{\vr} \right) + \Grad p(\vr) &= 0.
\end{split}
\end{equation}
Here, the well known problem is that there are basically no {\it a priori} bounds for the velocity itself but rather for the momentum $\vc{m}$. To recover $\vu$, a lower bound on $\vr$ must be available. We will discuss this issue later in Section \ref{S:positive}.

\subsubsection{Weak formulation}

The \emph{weak formulation} of problem (\ref{M1}), (\ref{M2}) written in the conservative variables reads:
\begin{equation} \label{M3}
\begin{split}
\left[ \intO{ \vr \varphi } \right]_{t = 0}^{t = \tau} &=
\int_0^\tau \intO{ \left[ \vr \partial_t \varphi + \vc{m} \cdot \Grad \varphi \right] } \dt\\
\mbox{for any}\ \tau &\in [0,T],\  {{\varphi \in C^1([0,T] \times \Omega)}};\\
\left[ \intO{ \vc{m} \cdot \bfphi } \right]_{t = 0}^{t = \tau} &=
\int_0^\tau \intO{ \left[ \vc{m} \cdot \partial_t \bfphi + \frac{ \vc{m} \otimes \vc{m} }{\vr} : \Grad \bfphi + p(\vr) \Div \bfphi \right] } \dt\\
\mbox{for any}\ \tau &\in [0,T],\
{{ \bfphi \in C^1([0,T] \times \Omega; R^N).
}}
\end{split}
\end{equation}

\begin{Remark} \label{R1}

Note that the weak formulation (\ref{M3}) already includes satisfaction of the initial conditions
\begin{equation} \label{M4}
\vr(0, \cdot) = \vr^0, \ \vc{m}(0, \cdot) = \vc{m}^0.
\end{equation}

\end{Remark}

Let
\begin{align}\label{ppot}
P(\vr) \equiv \vr \int_1^\vr \frac{p(z)}{z^2} \ {\rm d}z
\end{align}
be the so--called pressure potential.
The weak formulation (\ref{M3}), (\ref{M4}) is usually supplemented by the \emph{energy inequality}
\begin{equation*} 
\begin{split}
&\left[ \intO{ \left( \frac{1}{2} \frac{ |\vc{m}|^2 }{\vr} + P(\vr) \right) \varphi }  \right]_{t=0}^{t = \tau} \\ &\leq
\int_0^\tau \intO{ \left[ \left( \frac{1}{2} \frac{ |\vc{m}|^2 }{\vr} + P(\vr) \right) \partial_t \varphi
+ \left( \frac{1}{2} \frac{ |\vc{m}|^2 }{\vr} + P(\vr) \right) \frac{\vc{m}}{\vr} \cdot \Grad \varphi
+ p(\vr) \frac{\vc{m}}{\vr} \cdot \Grad \varphi
\right] } \dt
\end{split}
\end{equation*}
for a.a. $\tau \in [0,T]$ and any {{$\varphi \in C^1([0,T] \times \Omega)$, $\varphi \geq 0$.}}

It is easy to deduce, taking $\varphi \equiv 1$ in the first equation in (\ref{M3}), that the total mass,
\[
\intO{ \vr(\tau, \cdot) } = \intO{ \vr^0 },\ \tau \in [0,T]
\]
is a conserved quantity. In particular, one may replace $P$, given by (\ref{ppot}), by
\[
\frac{a}{\gamma - 1} \vr^\gamma
\]
in the energy inequality as long as the flow is isentropic.

\subsubsection{Measure--valued solutions}

The concept of \emph{measure--valued solution} to (\ref{M1c}) was introduced by Gwiazda, \'Swierczewska-Gwiazda, and
Wiedemann \cite{GSWW} in the framework of Alibert and Bouchitt\'e \cite{AliBou}.
There is also a general framework for hyperbolic system admitting $L^\infty-$ {\it a priori} bounds by
Brenier et al. \cite{BrDeSz}.
Here, we prefer a simpler and more versatile approach proposed
in \cite{FGSWW1}. Although the measure valued solutions are generally thought of as Young measures, with the associated concentration defect,
associated to sequences of approximate/exact solutions, we do not insist on this interpretation and introduce (DMV) solutions  as objects independent of any approximating sequence.

\begin{Definition} \label{D1}

Let
\[
\mathcal{F}= \left\{ [\vr, \vc{m} ] \ \Big| \ \vr \geq 0, \ \vc{m} \in R^N \right\}.
\]
We say that a parametrized family of probability measures $\left\{ \mathcal{V}_{t,x} \right\}_{t \in (0,T),\ x \in \Omega}$ defined
on the space $\mathcal{F}$ is a \emph{dissipative measure--valued (DMV) solution} of problem (\ref{M1}) with the initial conditions
\[
\mathcal{V}_{0,x} \in \mathcal{P}(\mathcal{F}),
\]

$\mathcal{P}$ denoting the set of (Borel) probability measures,
if
\begin{itemize}
\item
\[
(t,x) \mapsto \mathcal{V}_{t,x} \ \mbox{is weakly-(*) measurable mapping from the physical space}\ (0,T) \times \Omega \ \mbox{into}\
\mathcal{P}(\mathcal{F});
\]
\item
\begin{equation} \label{M6}
\begin{split}
\left[ \intO{ \left< \mathcal{V}_{t,x}; \vr \right> \varphi } \right]_{t = 0}^{t = \tau} &=
\int_0^\tau \intO{ \left[ \left< \mathcal{V}_{t,x}; \vr \right> \partial_t \varphi + \left< \mathcal{V}_{t,x}; \vc{m} \right> \cdot \Grad \varphi \right] }
\dt
+ \int_0^\tau \int_{{\Omega}} \Grad \varphi \cdot {\rm d} \mu^1_C
\\
\mbox{for a.a.}\ \tau &\in (0,T),\  {{\varphi \in C^1([0,T] \times \Omega)}}
\end{split}
\end{equation}
\begin{equation*} 
\begin{split}
&\left[ \intO{ \left< \mathcal{V}_{t,x} ; \vc{m} \right> \cdot \bfphi } \right]_{t = 0}^{t = \tau} \\ &=
\int_0^\tau \intO{ \left[ \left< \mathcal{V}_{t,x} ;\vc{m} \right> \cdot \partial_t \bfphi + \left< \mathcal{V}_{t,x} ; \frac{ \vc{m} \otimes \vc{m} }{\vr}
\right> : \Grad \bfphi + \left< \mathcal{V}_{t,x}; p(\vr) \right> \Div \bfphi \right] } \dt\\
&+ \int_0^\tau \int_{{\Omega}} \Grad \varphi : {\rm d} \mu^2_C\\
&\mbox{for a.a.}\ \tau \in (0,T),\  {{ \bfphi \in C^1([0,T] \times \Omega; R^N),}}
\end{split}
\end{equation*}
where
\[
\mu^1_C \in \mathcal{M}([0,T] \times \Omega; R^N),\
\mu^2_C \in \mathcal{M}([0,T] \times \Omega; R^{N \times N})
\]
are signed vector--valued \emph{concentration measures} defined on the physical space {{$[0,T] \times \Omega$;
}}
\item the energy inequality
\begin{equation} \label{M8}
\left[ \intO{ \left< \mathcal{V}_{t,x} ; \left( \frac{1}{2} \frac{ |\vc{m}|^2 }{\vr} + P(\vr)  \right) \right> } \right]_{t = 0}^{t = \tau} \leq 0
\end{equation}
holds for a.a. $\tau \in (0,T)$;
we denote
\[
\mathcal{D} (\tau) \equiv - \left[ \intO{ \left< \mathcal{V}_{t,x} ; \left( \frac{1}{2} \frac{ |\vc{m}|^2 }{\vr} + P(\vr)  \right) \right>  } \right]_{t = 0}^{t = \tau}
\]
the \emph{dissipation defect} - a non-negative $L^\infty$ function;

\item the {dissipation defect}
dominates the concentration measures $\mu^1_C$, $\mu^2_C$:
\begin{equation} \label{M9}
\int_{{\Omega}} 1 \ {\rm d} |\mu^1_C| \ + \int_{{\Omega}} 1 \ {\rm d} |\mu^2_C| \aleq \mathcal{D} \ \mbox{a.a. in}\ (0,T).
\end{equation}
\end{itemize}

Here and hereafter the symbol $A \aleq B$ means $A \leq cB$ for a generic positive constant $c$.

\begin{Remark} \label{R2}

The precise meaning of (\ref{M9}) is
\[
\sup_{ \| \bfphi \|_{C({\Omega}; R^N)}  \leq 1 } \int_0^T \int_{{\Omega}} \psi \bfphi \cdot {\rm d} \mu^1_C
+ \sup_{ \| \bfphi \|_{C({\Omega}; R^{N \times N})}  \leq 1 } \int_0^T \int_{{\Omega}} \psi \bfphi : {\rm d} \mu^2_C
\aleq \int_0^T \mathcal{D} \psi \dt
\]
for any $\psi \in C[0,T]$, $\psi \geq 0$. Relation (\ref{M9}) can be replaced by a weaker stipulation
\[
\int_0^\tau \int_{{\Omega}} 1 \ {\rm d} |\mu^1_C| \ + \int_0^\tau \int_{{\Omega}} 1 \ {\rm d} |\mu^2_C| \aleq \int_0^\tau \mathcal{D} \ \dt \  \ \mbox{any}\ \tau \in (0,T).
\]

\end{Remark}

\begin{Remark} \label{R3}

We tacitly assume that all expressions in (\ref{M6}--\ref{M8}) are at least integrable on the physical space $(0,T) \times \Omega$.

\end{Remark}

\end{Definition}

The key result is the (DMV)--strong uniqueness principle shown in Gwiazda et al. \cite{GSWW}, and also in \cite{FGSWW1}:

\begin{Proposition} \label{P1}

Let the initial data $\{ \mathcal{V}_{0,x} \}_{x \in \Omega}$ be given as
\[
\mathcal{V}_{0,x} = \delta_{ \vr^0(x), \vc{m}^0(x) } \ \mbox{for a.a.}\ x \in \Omega;
\]
where
\[
\vr^0 \in C^1 ({\Omega}), \ \vc{m}^0 \in C^1 ({\Omega}; R^N), \ \vr^0(x) > 0 \ \mbox{for all}\ x \in {\Omega}.
\]
Suppose that the problem (\ref{M1}), (\ref{M2}) admits a strong solution $\vr \in C^1([0,T] \times {\Omega})$, $\vc{m} \in C^1([0,T] \times {\Omega}; R^N)$
defined in $[0,T]$,
with the initial data $\vr^0$, $\vc{m}^0$. Let $\{ \mathcal{V}_{t,x} \}_{t \in (0,T), x \in \Omega}$ be a (DMV) solution of the same problem
in the sense specified in Definition \ref{D1}, with the initial data $\mathcal{V}_{0,x}$.

Then
\[
\mathcal{V}_{t,x} = \delta_{\vr(t,x), \vc{m}(t,x)} \ \mbox{for a.a.}\ (t,x) \in (0,T) \times \Omega.
\]

\end{Proposition}

\begin{Remark}

Strictly speaking Proposition~\ref{P1} was originally proved on a smooth bounded domain with the impermeability condition
\[
\vu \cdot \vc{n}|_{\partial \Omega} = 0.
\]
However, it can be easily checked that the proof applies to the periodic boundary conditions with only obvious modifications.

\end{Remark}

\subsection{Measure--valued solutions for the complete Euler system}
\label{F}

Similarly to the preceeding section, we may introduce (DMV) solutions for the
\emph{complete Euler system}
\begin{equation} \label{F1}
\begin{split}
\partial_t \vr + \Div (\vr \vu) &= 0, \\
\partial_t (\vr \vu) + \Div (\vr \vu \otimes \vu) + \Grad p(\vr, \vt) &= 0,\\
\partial_t \left( \frac{1}{2} \vr |\vu|^2 + \vr e(\vr, \vt) \right)
+ \Div \left[ \left( \frac{1}{2} \vr |\vu|^2 + \vr e(\vr, \vt) \right) \vu \right]
+ \Div (p(\vr, \vt) \vu) &= 0
\end{split}
\end{equation}
supplemented with  the \emph{periodic boundary conditions}, meaning $\Omega$ can be identified with the flat torus
\begin{equation} \label{F2}
\Omega =\big( [0,1]|_{\{0,1\}}\big)^N.
\end{equation}

Here, the new variable is the absolute temperature $\vt$, $e = e(\vr, \vt)$ is the specific internal energy, and
the third equation in (\ref{F1}) expresses the conservation of the total energy. In addition, we suppose that
$p$ and $e$ are interrelated to the \emph{specific entropy}  $s = s(\vr, \vt)$ via Gibbs' equation
\begin{equation} \label{GIB}
\vt Ds = De + P D\left( \frac{1}{\vr} \right).
\end{equation}
Accordingly, if all quantities in (\ref{F1}) are smooth, the entropy satisfies a transport equation
\[
\partial_t (\vr s) + \Div (\vr s \vu) = 0.
\]
In the context of \emph{weak} solutions, the entropy balance is replaced by an inequality
\begin{equation*} 
\partial_t (\vr s) + \Div (\vr s \vu) \geq 0
\end{equation*}
that may be seen as a mathematical formulation of the Second law of thermodynamics.

Similarly to the preceding section,
the concept of (DMV) solution uses the conservative variables:
the density $\vr$, the momentum $\vc{m} = \vr \vu$, and the total energy $E = \frac{1}{2} \vr |\vu|^2 + \vr e(\vr, \vt)$. In addition, we suppose
a relation between the pressure and the internal energy,
\begin{equation} \label{F4}
p = (\gamma - 1) \vr e , \ \mbox{with}\ \gamma > 1.
\end{equation}
Under these circumstances, we have
\[
s= S \left( \frac{(\gamma - 1) e}{\vr^{\gamma - 1}} \right) =
S \left( \frac{p}{\vr^{\gamma}} \right)
\]
for a certain function $S$. Accordingly, the system (\ref{F1}) rewrites as
\begin{equation} \label{F5}
\begin{split}
\partial_t \vr + \Div \vc{m} &= 0,\\
\partial_t \vc{m} + \Div \left( \frac{\vc{m} \otimes \vc{m}}{\vr} \right) + (\gamma - 1) \Grad \left( E - \frac{1}{2} \frac{|\vc{m}|^2}{\vr} \right)  &= 0,\\
\partial_t E +
\Div \left[ \left( E + (\gamma - 1) \left( E - \frac{1}{2} \frac{|\vc{m}|^2}{\vr} \right) \right) \frac{\vc{m}}{\vr} \right] &= 0,
\end{split}
\end{equation}
together with the associated entropy inequality
\begin{equation} \label{F6}
\partial_t \left( \vr S \left( (\gamma - 1) \frac{E - \frac{1}{2} \frac{|\vc{m}|^2}{\vr} }{\vr^\gamma} \right) \right)
+ \Div \left[ S \left( (\gamma - 1) \frac{E - \frac{1}{2} \frac{|\vc{m}|^2}{\vr} }{\vr^\gamma} \right) \vc{m} \right] \equiv \sigma \geq 0.
\end{equation}
In addition, we may use, formally, the equation of continuity, to replace (\ref{F6}) by a more restrictive stipulation
\begin{equation} \label{F7}
\partial_t \left( \vr \mathcal{S}_\chi \left( (\gamma - 1) \frac{E - \frac{1}{2} \frac{|\vc{m}|^2}{\vr} }{\vr^\gamma} \right) \right)
+ \Div \left[ \mathcal{S}_\chi \left( (\gamma - 1) \frac{E - \frac{1}{2} \frac{|\vc{m}|^2}{\vr} }{\vr^\gamma} \right) \vc{m} \right] \equiv \sigma \geq 0,
\end{equation}
where
\begin{align}\label{S_chi}
\mathcal{S}_\chi = \chi \circ S,\ \chi: R \to R \ \mbox{an increasing concave function}\ \chi \leq \Ov{\chi}.
\end{align}
Inequality (\ref{F7}) may be seen as a \emph{renormalized} variant of (\ref{F6}).
For the sake of simplicity, we focus on the constitutive equations of a perfect gas,
specifically
\begin{align}\label{pressure_cE}
p(\vr,\vt)=\vr\vt, \ e(\vr,\vt)=c_v\vt, \ s(\vr,\vt)=\log\left(\frac{\vt^{c_v}}{\vr}\right),
\end{align}
where $c_v = \frac{1}{\gamma - 1}$ is the (constant) specific heat at constant volume. Consequently,
\begin{align}\label{S_cE}
S\left(Z\right)=\frac{1}{\gamma-1}\log\left(Z\right),\ \mbox{and entropies}\ \eta=\vr\chi\left(\frac{1}{\gamma-1}\log\left(\frac{p}{\vr^{\gamma}}\right)\right)
\end{align}
for $\chi$ as in \eqref{S_chi}.
We are ready to state the definition of a (DMV) solution for the complete Euler system (\ref{F5}), (\ref{F2}), cf. \cite{BF}.

\begin{Definition} \label{D2}
Let
\[
\mathcal{F}= \left\{ [\vr, \vc{m}, E ] \ \Big| \ \vr \geq 0, \ \vc{m} \in R^N, E \geq 0 \right\}.
\]
We say that a parameterized family of probability measures $\left\{ \mathcal{V}_{t,x} \right\}_{t \in (0,T),\ x \in \Omega}$ defined
on the space $\mathcal{F}$ is a \emph{dissipative measure--valued (DMV) solution} of problem (\ref{F5}), (\ref{F2}) with the initial conditions
\[
\mathcal{V}_{0,x} \in \mathcal{P}(\mathcal{F})
\]
if
\begin{itemize}
\item
\[
(t,x) \mapsto \mathcal{V}_{t,x} \ \mbox{is weakly-(*) measurable mapping from the physical space}\ (0,T) \times \Omega \ \mbox{into}\
\mathcal{P}(\mathcal{F});
\]
\item
\begin{equation*}
\int_0^T \intO{ \left[ \left< \mathcal{V}_{t,x}; \vr \right> \partial_t \varphi + \left< \mathcal{V}_{t,x}; \vc{m} \right> \cdot
\Grad \varphi \right] } \dt = - \intO{ \left< \mathcal{V}_{0,x} ; \vr \right> \varphi(0, \cdot)}
\end{equation*}
for any $\varphi \in C^1_c([0,T) \times \Omega)$;
\item
\begin{equation*} 
\begin{split}
\int_0^T & \intO{ \left[ \left< \mathcal{V}_{t,x}; \vc{m} \right> \cdot \partial_t \bfphi  + \left< \mathcal{V}_{t,x}; \frac{ \vc{m} \otimes \vc{m} }{\vr}
\right> : \Grad \bfphi + (\gamma - 1) \left< \mathcal{V}_{t,x}; E - \frac{1}{2} \frac{|\vc{m}|^2 }{\vr} \right> \Div \bfphi \right] }\dt \\
&= - \intO{ \left< \mathcal{V}_{0,x}; \vc{m} \right> \cdot \bfphi(0, \cdot) } +
\int_0^T \int_\Omega \Grad \bfphi : {\rm d} \mu_C
\end{split}
\end{equation*}
for any $\bfphi \in C^1_c([0, T) \times \Omega; R^N)$,
 where $\mu_C$ is a (vectorial) signed measure on $[0,T]\times \Omega$;
\item
\begin{equation*} 
\intO{ \left< \mathcal{V}_{\tau,x}; E \right> } \leq \intO{ \left< \mathcal{V}_{0,x}; E \right> } \ \mbox{for a.a.}\ \tau \in (0,T);
\end{equation*}
\item
\begin{equation*} 
\begin{split}
\int_0^T &\intO{ \left[ \left< \mathcal{V}_{t,x} ; \vr\mathcal{S}_\chi (\vr, \vc{m}, E) \right> \partial_t \varphi +
\left< \mathcal{V}_{t,x}; \mathcal{S}_\chi (\vr, \vc{m}, E) \vc{m} \right> \cdot \Grad \varphi  \right] } \dt \\
&\leq - \intO{ \left< \mathcal{V}_{0,x}; \vr\mathcal{S}_\chi(\vr, \vc{m}, E) \right> \varphi(0, \cdot) }
\end{split}
\end{equation*}
for any $\varphi \in C^1_c([0,T) \times \Omega)$, $\varphi \geq 0$, and any
$\chi$ defined on $R,$ increasing, concave, $\chi(Z) \leq \Ov{\chi}$  for all $ Z$;

\item
\begin{equation*} \label{F12}
\int_0^\tau \int_\Omega d \left| \mu_C \right|  \leq c(N, \gamma) \int_0^\tau \intO{ \left[ \left< \mathcal{V}_{0,x}; E \right> - \left< \mathcal{V}_{t,x}; E \right> \right]}
\dt \ \mbox{for any}\ 0 \leq \tau < T.
\end{equation*}
\end{itemize}
\end{Definition}

Finally, we formulate an analogue of the weak--strong uniqueness result stated in Proposition~\ref{P1}. To this end, we recall the hypothesis of thermodynamic stability:
\begin{equation} \label{F13}
\frac{\partial p(\vr, \vt) }{\partial \vr} > 0, \ \frac{\partial e(\vr, \vt) }{\partial \vt} > 0 \ \mbox{for all}\ \vr, \vt > 0,
\end{equation}
or, in terms of the conservative variables,
\[
\begin{split}
(\vr, \vc{m}, E) \mapsto \vr S \left( (\gamma - 1) \frac{E - \frac{1}{2} \frac{|\vc{m}|^2}{\vr} }{\vr^\gamma} \right)\ \mbox{is a concave upper semi--continuous function on}\ \mathcal{F},
\end{split}
\]
see \cite{BreFei17B} for details.
\begin{Remark}\label{R4}
It follows from \cite{BreFei17B} that the entropy $\eta=\vr\mathcal{S}_{\chi}$ with $\mathcal{S}_{\chi}$ as in \eqref{S_chi} is concave for any function $S$ satisfying
\begin{align*}
(\gamma-1)S'(Z)+\gamma S''(Z)Z <0 \ \mbox{for all} \ Z >0.
\end{align*}
In particular for $S$ in \eqref{S_cE}.
\end{Remark}

We are ready to state the weak--strong uniqueness result, see \cite[Theorem 3.3]{BF}.

\begin{Proposition} \label{P2}\ \\
Let the thermodynamic functions $p$, $e$, and $s$ satisfy the hypotheses (\ref{GIB}), (\ref{F4}), (\ref{F13}). Suppose that the Euler system (\ref{F1}),
(\ref{F2}) admits a continuously differentiable
solution $(\tvr, \tvt, \tvu)$ in $[0,T] \times \Omega$ emanating from the initial data
\[
\tvr^0 > 0,\ \tvt^0 > 0 \ \mbox{in}\ {\Omega}.
\]

Assume that  $\{ \mathcal{V}_{t,x} \}_{(t,x) \in (0,T) \times \Omega}$ is a (DMV) solution of the  system (\ref{F5}), (\ref{F2}) in the sense specified in Definition \ref{D2}, such that
\[
{\mathcal{V}}_{0,x} = \delta_{\tvr^0(x), \tvr^0 \tvu^0 (x), \frac{1}{2} \tvr^0(x) |\tvu^0(x)|^2 + \tvr^0 e(\tvr^0, \tvt^0)(x)} \ \mbox{for a.a.}\ x \in \Omega.
\]

Then
\[
\mathcal{V}_{t,x} = \delta_{\tvr(t,x), \tvr \tvu (t,x), \frac{1}{2} \tvr(x) |\tvu(x)|^2 + \tvr e(\tvr, \tvt)(t,x)} \ \mbox{for a.a.}\ (t,x) \in (0,T) \times \Omega.
\]

\end{Proposition}

\section{Entropy stable finite volume schemes for conservation laws}\label{S:scheme}

We start with recalling the concept of entropy stable finite volume schemes for a general multidimensional system of hyperbolic conservation laws
\begin{equation}\label{conservative_system}
\begin{aligned}
\pd{\vU}{t}+\di\bff(\vU)&=0,  & & \mbox{ in } \Omega \times (0,T)\\
\vU(0,\cdot)&=\vU^0, && \mbox{ in } \Omega.
\end{aligned}
\end{equation}
Here $\vU,$ $\bff(\vU)$ denote the vectors of conservative variables and the flux function, respectively. The system (\ref{conservative_system}) is usually accompanied with suitable boundary conditions. As agreed above, we will exclusively use the
periodic boundary conditions.  Throughout the paper we will
confine ourselves to semi-discrete schemes.\, Specifically, the time will remain continuous, the discretization applied to
the space variable only. The question of time discretization is more subtle. As is well-known the \emph{implicit} time discretization gives rise to the entropy production and thus the correct sign in the entropy inequality. Consequently, the resulting fully implicit scheme will be entropy stable once its semi-discrete variant was entropy stable. On the other hand, the \emph{explicit} time discretization which is a natural choice for hyperbolic conservation laws may actually reduce the (physical) entropy, and the interplay between the spatial entropy production and temporal entropy dissipation has to be taken into account in practical applications, see, e.g., \cite{Tad03, LMR02, B03}.

\subsection{Spatial discretization}

The relevant domain for the space discretization
is $\Omega\equiv\Omega_h\subset R^N,$ $N=1,2,3,$ where $\Omega_h\defeq [0,\ell]^N,$ $\ell>0,$ being divided into finite volume cells $K$, i.e.,
\begin{align*}
\Ov{\Omega}_h \defeq \bigcup_{K\in\grid} \Ov{K}.
\end{align*}
Mesh $\grid$ is a regular quadrilateral grid.
For instance, in two space dimensions,  cell $\elem,$ its center $S_K,$ and the uniform mesh size $h$ are given by
\begin{align*}
\elem\defeq \left[\dvxp{x},\dvxm{x}\right)\times \left[\dvyp{y},\dvym{y}\right), \quad S_K\defeq (x_i,y_j)= \left(\frac{\dvxm{x}+\dvxp{x}}{h},\frac{\dvym{y}+\dvyp{y}}{h} \right),
\end{align*}
and $\displaystyle h\defeq \dvxp{x}-\dvxm{x}=\dvyp{y}-\dvym{y}, $
respectively.
\begin{Remark}\label{R44}
 Note that the usual relabeling $(x_1,x_2)\mapsto (x,y)$ has been taken into account in the above example. It is also possible to consider the rectangular cells with $h_x=ch_y,$ where $c$ is a positive constant and $h_x,$ $h_y$ are fixed mesh sizes in $x$- and $y$-direction, respectively. An analogous generalization of the three mesh sizes $h_x,$ $h_y,$ $h_z$ is applicable   for $N=3$ as well. For the sake of simplicity we keep the mesh size fixed in all space directions.
\end{Remark}
Let $ X(\grid)$ denote the space of piecewise constant functions defined on mesh $\grid.$ For $g_h \in X(\grid)$ we set
$\displaystyle g_K \equiv g_{h_{|_{\elem}}}.$  Then it holds that
\begin{align*}
\intO{g_h}=h^N\sum_{K\in\grid} g_K .
\end{align*}
Further, we define the projection
\begin{align*}
\Pi_h : L^1(\Omega) \rightarrow X(\grid), \quad (\Pi_h(\phi))_{K}\defeq\frac{1}{h^N}\int_{\elem}{\phi(x)\,dx}.
\end{align*}
Boundary $\partial K$ of a cell $\elem$  is created by faces $\sigma.$ The face between two neighbouring cells $\elem$ and $\elemL$ shall be denoted by $\sigma=K|L.$
  By $\edge$ we denote the set of all faces $\sigma$ of all cells $\elem \in \grid.$
 The value of  $G_h $ on the face $\sigma$ shall be denoted by $\bv{G},$ and analogously for faces $\bndr$ of cell $K$ in $\pm\bfe_s$ direction. Note that $\bfe_s$ is the unit basis vector in the $s$-th space direction, $s =1, \dots, N.$ For $g_h, G_h \in X(\grid)$ we define the following discrete operators
\begin{align*}
\der{g_h}&\defeq \frac{g_L-g_J}{2h}, \ \derp{g_h}\defeq \frac{g_L-g_K}{h},\ \derm{g_h}\defeq \frac{g_K-g_J}{h},
\quad L=K+h\bfe_s, J=K-h\bfe_s, \\
\fluxder{G_h}&\defeq \frac{G_{\sigma,s+}-G_{\sigma,s-}}{h}, \quad s =1, \dots, N.
\end{align*}
Let $\mathcal{N}(K)$  denote the set of all neighbouring cells of the cell $\elem$. The discrete  Laplace and divergence operators
are defined as follows
\begin{align*}
 \laph{g_h}&\defeq \frac{1}{h^2}\sum_{L\in\mathcal{N}(K)}(g_L-g_K)=\sum_{s=1}^N\laphs{g_h}, \\  \diht{\bfg_h}&\defeq \sum_{s=1}^N \der{g_h^s}, \quad
\dih{\bfG_h} \defeq \sum_{s=1}^N \fluxder{G_h^s}.
\end{align*}
Furthermore, on the face $\sigma=K|L \in \edge$ we define the  jump and mean value operators
\begin{align*}
\jump{g_h}:=g_L \bold{n}_K^+ + g_K\bold{n}_K^-, \quad \avrg{g_h} \defeq  \frac{g_K+g_L}{2}, \quad L=K+h\bfe_s,\ s=1,\ldots,N,
\end{align*}
respectively. Here $\mathbf{n}_K^+,$ $\mathbf{n}_K^- \equiv \mathbf{n}_L^+$ denote the unit  outer normal to $\elem$ and $\elemL,$ respectively. Note that in our case the mesh is a regular  quadrilateral grid, and thus $\mathbf{n}_K^{\pm} || \mathbf{e}_s $ for some $s=1,\ldots,N.$ Finally, we introduce the mean value of $g_h \in X(\grid)$ in cell $\elem$ in the direction of $\bfe_s$ by
\begin{align*}
\avrgK{g_h}:=\frac{g_L+g_J}{2}, \quad L=K+h\bfe_s, \ J=K-h\bfe_s.
\end{align*}

\subsection{Entropy stable numerical scheme}

By $\vU_h(t) \in X(\grid)^{M},$ $M> 1,$ we denote
the  solution of a semi-discrete finite volume scheme
\begin{equation}\label{scheme}
\begin{aligned}
\Dt{\vU_K(t)}+\dih{\bfF_h(t)}=0,  \quad t>0, \quad K \in \grid,\\
\vU_K(0)=\dv{(\Pi_h(\vU^0))}, \quad K \in \grid.
\end{aligned}
\end{equation}
Recall that $\vU_h(t)_{|_{\elem}}=\vU_K(t)$
is the value of  finite volume approximation $\vU_h(t)$ in cell $\elem.$
The numerical flux function $\bfF_h$  quantifies the flux
across the interfaces $\sigma\in\edge.$  For $\sigma=K|L$ we have
$\bv{\bfF}\equiv\bfF_h(\vU_K,\vU_L).$  In what follows we formulate assumptions on admissible numerical fluxes.

Firstly, the numerical flux $\bfF_h$ is assumed to be  \emph{consistent} with the physical flux $\bff$ in the sense that $\bfF_h(\vw,\vw)=\bff(\vw)$ for all $\vw\in R^M.$ Moreover, it is assumed to be  \emph{locally Lipschitz continuous}, i.e., for  every compact set $D\subset R^M$ there exists a $C>0$ such that
 \begin{align*} 
 \left\|\bfF_{\sigma}(t)-\bff(\vU_K(t))\right\|\equiv\left\|\bfF_h(\vU_K(t),\vU_L(t))-\bff(\vU_K(t))\right\|\leq C\|\vU_K(t)-\vU_L(t)\|, \ \sigma=K|L,
 \end{align*}
 whenever $\vU_K(t),$ $\vU_L(t) \in D$ for  $t \in [0,T].$ Note that all numerical fluxes discussed below are consistent and locally Lisphitz continuous.\\

The discrete entropy inequality plays a crucial role in obtaining stability results for $\vU_h(t).$
Let $(\eta,\bfq)$ be an \emph{entropy pair} associated with system \eqref{conservative_system}, i.e., $(\eta,\bfq): R^M \rightarrow R\times  R^N$ such that  $\eta$ is concave and $\bfq$ satisfies for all $\vw\in R^M$   the compatibility condition
\begin{align*}
\nabla_{\vw}q^s(\vw)^T=\nabla_{\vw}\eta(\vw)^T\nabla_{\vw} f^s(\vw),\quad  s=1,\ldots,N.
\end{align*}
Scheme \eqref{scheme} is then said to be \textit{entropy stable} if it satisfies the discrete entropy inequality
\begin{align}\label{dis_en_ineq}
\Dt{{\eta(\vU_K}(t))}+\dih{\bfQ_h(t)} \geq 0, \quad K\in \grid, \ t>0.
\end{align}
If, in particular, equality holds in \eqref{dis_en_ineq}, we say the scheme \eqref{scheme} is \textit{entropy conservative}.
Here $\bfQ_h$ denotes the numerical entropy flux function that is  a function of two neighbouring values, i.e., $\bv{\bfQ}\equiv \bfQ_h(\vU_K,\vU_L)$ for $\sigma=K|L.$
 It is assumed to be \emph{consistent} with the differential entropy flux $\bfq,$ i.e., $\bfQ_h(\vw,\vw)=\bfq(\vw)$ for all $\vw\in R^M.$ Following the work of Tadmor et al. \cite{Tad03,FMT16}, entropy flux $\bfQ_h$ can be explicitly written in terms of the vector of entropy variables $\vV$, the numerical flux $\bfF_h$ and the potential function $\psi=\psi(\vU(\vV)),$ as
\begin{align}\label{def_Q}
\bv{\bfQ}\defeq \avrg{\vV_h}\bv{\bfF}-\avrg{\psi(\vV_h)}.
\end{align}
We shall omit the dependence on time  whenever there is no confusion.
Further, we say that  solution $\vU_h(t)$ of  scheme \eqref{scheme} satisfies the {\it weak BV (bounded variation)} condition if
\begin{align}\label{weakBV1}
\intOT{\sum_{\sigma\in\edge} \bv{\lambda} \big|\jump{\vU_h(t)}\big| h^N} \to 0 \quad \mbox { as } \quad h \to 0^+,
\end{align}
where $\bv{\lambda}$ is the coefficient of numerical viscosity that will be introduced in (\ref{num_flux}).


\begin{Remark}\label{R5}
In the literature (mathematical) convex entropy, $-\eta$, is often used,  see, e.g., \cite{Tad03,FMT16}. Here we  prefer to work with (physical) entropy that is a concave function on its effective domain, cf. Remark~\ref{R4}.
\end{Remark}

\begin{Remark}
For the complete Euler system \eqref{F5} the vector of entropy variables is given in terms of conservative variables $\vU$ by
\begin{align*}
\vV\defeq\nabla_{\vU}\eta(\vU)=\frac{\chi'(S(\vU))}{p}\left( \begin{array}{c}
E+\frac{p}{\gamma-1}\big((\gamma-1)\frac{\chi(S(\vU))}{\chi'(S(\vU))}-\gamma-1\big) \\
-\vm \\
\vr \\
\end{array}\right). \
\end{align*}
Substituting for pressure $\displaystyle p=(\gamma-1)\left(E-\frac{1}{2}\frac{|\vm|^2}{\vr}\right)$ we obtain
\begin{align*}
\vV=\frac{\chi'(S(\vU))}{(\gamma-1)\left(E-\frac{1}{2}\frac{|\vm|^2}{\vr}\right)}\left( \begin{array}{c}
E\left((\gamma-1)\frac{\chi(S(\vU))}{\chi'(S(\vU))}-\gamma\right)-\frac{1}{2}\frac{|\vm|^2}{\vr}\big((\gamma-1)\frac{\chi(S(\vU))}{\chi'(S(\vU))}-\gamma-1\big) \\
-\vm \\
\vr \\
\end{array}\right).
\end{align*}
The potential function for the complete Euler system reads
$
\psi(\vU(\vV))=-\chi'(S(\vU))\vm.
$
For the barotropic Euler system the corresponding entropy variables and entropy potential are given by
\begin{align*}
\vV=\left(\begin{array}{c}
\frac{a\gamma}{\gamma-1}\vr^{\gamma-1}-\frac{|\vm|^2}{2\vr^2} \\
\frac{\vm}{\vr}
\end{array}\right), \qquad \psi(\vU(\vV))= a \gamma \vr^{\gamma-1}\vm.
\end{align*}

The specific form of $\vc{V}$, as well as the flux function used in the discretization of the complete Euler system discussed below, immediately reveals a
peculiar difficulty connected with the development of the vacuum state $\vr = 0$ in finite time. Indeed the fluxes are not correctly defined as soon as $\vr = 0$, while the corresponding Lipschitz constant may blow up for $\vr \to 0$. We discuss this problem in Section \ref{S:positive} below.

\end{Remark}

\subsubsection{Examples of entropy stable numerical schemes}

\begin{itemize}
\item Rusanov / Lax-Friedrichs schemes\\
Following \cite{Tad03} the Rusanov scheme with the following numerical flux is entropy stable.
\begin{equation*}
\bv{\bfF}\defeq \avrg{\bff(\vU_h)}- d_\sigma \jump{\vU_h},
\end{equation*}
where $\displaystyle d_\sigma = \dfrac{1}{2}\max_{s=1,\dots,N}( |\lambda^s(\vU_K)|, | \lambda^s(\vU_L)|  )$, $\sigma = K|L$ and  $\lambda^s$ is the $s-$th eigenvalue of the corresponding Jacobian matrix $\bff'(\vU_h)$. In the case that $ \displaystyle d_\sigma = \dfrac{1}{2} \max_{s=1,\dots,N}\max_{K \in \grid}|\lambda^s(\vU_K)|$ we obtain the Lax-Friedrichs scheme that is entropy stable, too.
\item  entropy stable Roe scheme \\
The following entropy stable version of the Roe scheme has been proposed in \cite{Tad03}
\begin{equation*}
\bv{\bfF}\defeq \avrg{\bff(\vU_h)}- D_\sigma \jump{\vU_h}.
\end{equation*}
Denoting $\Ov A_{\sigma}$ the Roe matrix, that satisfies
$\jump{\bfF} \equiv \Ov A_{\sigma} \jump{\vU_h} $, we define
the viscosity matrix $D_\sigma = d(\Ov A_{\sigma})$ with the function
$\displaystyle d(\Ov\lambda^s) = \max (|\Ov \lambda^s|, k C_\sigma \jump{\vU_h})$. Here $k >0$ is the upper bound
of $\frac{d^2\eta(\vU)}{d \vU^2}$ and $C_\sigma$ is chosen such that $\min_\lambda (\lambda (Q_\sigma)) \geq C_\sigma |\jump{\vV_h}|$,
$Q_\sigma$ is the viscosity matrix with respect  to the entropy variables $\vV_h$, see \cite{Tad03}, Theorem~5.3, Example~5.8.
\item Lax-Wendroff scheme \\
In \cite{FKMT17} the entropy stable Lax-Wendroff scheme has been presented. The numerical flux reads
\begin{equation*}
\bv{\bfF}:= \bv{\tilde{\bfF}^r} - d_\sigma | \jump{\vV_h}|^{r-1} \jump{\vV_h},
\end{equation*}
where $\bv{\tilde{\bfF}^r}$ is a $r-$th order entropy conservative numerical flux, see \cite{Tad03}, $d_\sigma$ is some positive number.
In \cite{fjdiss} it has been shown that this scheme is formally $r-$th order accurate, entropy stable and under the assumptions that
$\frac{d^2\eta(\vU)}{d \vU^2} \geq \Un{\eta} > 0$ (for convex mathematical entropy) and $d_\sigma \geq c > 0$ the scheme
satisfies the weak BV estimates (\ref{weakBV1}) with $\lambda_\sigma \equiv 1$.
\item TeCNO scheme \\
In \cite{FMT_2013} essential non-oscillatory entropy stable (TeCNO) schemes for system of conservation laws have been introduced. The numerical flux has the form
\begin{equation*}
\bv{\bfF} := \bv{\tilde{\bfF}^r} - \frac  1 2 D_\sigma  (\vV_{L}^- -  \vV_{K}^+),
\end{equation*}
where $\bv{\tilde{\bfF}^r}$ is a $r-$th order entropy conservative numerical flux as above,
 $D_\sigma$ is a positive definite matrix and
$\vV_{L}^-$,  $\vV_{K}^+$ are the cell interface values of a $r-$th order accurate ENO reconstruction. The scheme is formally $r-$th order accurate, entropy stable and satisfies weak BV estimates (\ref{weakBV1}) under the above mentioned assumptions on $\frac{d^2\eta(\vU)}{d \vU^2} $, see \cite{FMT_2013}, \cite{FKMT17}.
\end{itemize}

\subsection{Numerical schemes for the barotropic Euler system}

Our aim is to prove the convergence of some entropy stable finite volume schemes for the multidimensional Euler equations.
More precisely, we show that  the sequence of numerical solutions generate the Young measure that represents the dissipative measure-valued solution.  To illustrate the ideas we will consider scheme \eqref{scheme} with  a Lax-Friedrichs-type numerical flux $\bfF_h$ whose value on a face $\sigma = K|L$ is given by
\begin{align}\label{num_flux}
\bv{\bfF}\defeq \avrg{\bff(\vU_h)}-\bv{\lambda} \jump{\vU_h}.
\end{align}
Here the global diffusion coefficient is
 $\displaystyle \bv{\lambda} \equiv \lambda\defeq \max_{K\in\grid}\max_{s=1,\ldots,N} |\lambda^s(\vU_K)|$, while the local diffusion coefficient is
 $\displaystyle \bv{\lambda} \defeq \max_{s=1,\ldots,N} \max( |\lambda^s (\vU_K)|, |\lambda^s(\vU_L)|).$
 As already mentioned above $\lambda^s$ is the $s-$th eigenvalue of the corresponding Jacobian matrix $\bff'(\vU_h)$. Finite volume scheme with the local diffusion coefficient is also called in the literature the Rusanov scheme.

 Substituting  $\vU=[\vr,\vm]^T$ and $\bff(\vU)=[\vm, \frac{\vm\otimes\vm}{\vr}+p\mathbb{I}]^T,$ $p=a\vr^{\gamma},$  into \eqref{num_flux}    we derive the semi-discrete finite volume scheme for the barotropic Euler system:
 \begin{subequations}\label{scheme_bE1}
 \begin{align}
\Dt{\vr_K(t)}&+\diht{\vm_h(t)} - \dfrac{1}{h} \sum_{\sigma \in \partial K}  \bv{\lambda}  \jump{\vr_h(t)} ( \bfn_K^+\cdot\bfe_s) = 0, \label{cont_bE1} \\
\Dt{\vm_K(t)}&+\diht{\left(\frac{\vm_h(t)\otimes\vm_h(t)}{\vr_h(t)}+p_h(t)\mathbb{I}\right)} -\dfrac{1}{h} \sum_{\sigma \in \partial K} \bv{\lambda} \jump{\vm_h(t)} ( \bfn_K^+\cdot\bfe_s)= 0, \  t>0, \ K \in \grid. &\label{mom_bE1}
\end{align}
\end{subequations}
Note that $(\bfn_K^+ \cdot\bfe_s)$ determines whether the jump belongs to  in- or outgoing fluxes.
For the global numerical diffusion coefficient \eqref{num_flux} gives
\begin{subequations}\label{scheme_bE}
\begin{align}
\Dt{\vr_K(t)}&+\diht{\vm_h(t)}- \lambda h \laph{\vr_h(t)}=0, \label{cont_bE} &\\
\Dt{\vm_K(t)}&+\diht{\left(\frac{\vm_h(t)\otimes\vm_h(t)}{\vr_h(t)}+p_h(t)\mathbb{I}\right)}-\lambda  h\laph{\vm_h(t)}=0,\ t >0 \ K \in \grid. \label{mom_bE}&
\end{align}
Recall that  $p_h(t)=p(\vr_h(t))=a\vr_h^{\gamma}(t),$ $\gamma>1,$ $a>0,$  cf. \eqref{pressure_bE}.
The initial conditions for the schemes \eqref{scheme_bE1} and \eqref{scheme_bE} are prescribed as follows
\begin{align*}
&(\vr_K(0),\vm_K(0))^T=(\dv{(\Pi_h\vr^0)},\dv{(\Pi_h\vm^0)})^T,  \quad K \in \grid.
\end{align*}
\end{subequations}

\subsection{Numerical schemes for the complete Euler system}

Analogously as above, we insert  the corresponding vector of conservative variables $\vU=[\vr,\vm,E]^T$ and  the flux function $ \bff(\vU)=\left[\vm,\frac{\vm\otimes\vm}{\vr}+p\mathbb I, \frac{\vm}{\vr}(E+p)\right]^T,$ $p=(\gamma-1)(E-\frac{1}{2}\frac{|\vm|^2}{\vr}),$ into the definition of the Lax-Friedrichs-type numerical flux \eqref{num_flux} to obtain the finite volume scheme
\begin{subequations}\label{scheme_cE1}
\begin{align}
\Dt{\vr_K(t)}&+\diht{\vm_h(t)} - \dfrac{1}{h}  \sum_{\sigma \in \partial K} \bv{\lambda}  \jump{\vr_h(t)} ( \bfn_K^+\cdot\bfe_s) = 0, \label{cont_cE1} \\
\Dt{\vm_K(t)}&+\diht{\left(\frac{\vm_h(t)\otimes\vm_h(t)}{\vr_h(t)}+p_h(t)\mathbb{I}\right)} -\dfrac{1}{h} \sum_{\sigma \in \partial K} \bv{\lambda} \jump{\vm_h(t)}  ( \bfn_K^+\cdot\bfe_s) = 0,   \label{mom_cE1}\\
\Dt{E_K(t)}&+\diht{\left(\frac{\vm_h(t)}{\vr_h(t)}(E_h(t)+p_h(t))\right)}
- \dfrac{1}{h}  \sum_{\sigma \in \partial K} \bv{\lambda}  \jump{E_h(t)} ( \bfn_K^+\cdot\bfe_s) = 0,
\ t>0, \ K \in \grid.  \label{en_cE1}
\end{align}
\end{subequations}
 The global  numerical viscosity coefficient yields analogously as above
\begin{subequations}\label{scheme_cE}
\begin{align}
\Dt{\vr_K(t)}&+\diht{\vm_h(t)}-\lambda  h\laph{\vr_h(t)}=0, \label{cont_cE} \\
\Dt{\vm_K(t)}&+\diht{\left(\frac{\vm_h(t)\otimes\vm_h(t)}{\vr_h(t)}+p_h(t)\mathbb{I}\right)}-\lambda  h\laph{\vm_h(t)}=0, \label{mom_cE} \\
\Dt{E_K(t)}&+\diht{\left(\frac{\vm_h(t)}{\vr_h(t)}(E_h(t)+p_h(t))\right)}-\lambda  h\laph{E_h(t)}=0, \ t>0, \ K \in \grid.\label{en_cE}
\end{align}
Recall that $p_h(t)=(\gamma-1)\left(E_h(t)-\frac{1}{2}\frac{|\vm_h(t)|^2}{\vr_h(t)}\right).$
Finite volume schemes \eqref{scheme_cE1} and \eqref{scheme_cE} are  equipped with the initial conditions
\begin{align*}
&(\vr_K(0),\vm_K(0), E_K(0))^T=(\dv{(\Pi_h\vr^0)},\dv{(\Pi_h\vm^0)},\dv{(\Pi_h E^0)})^T,  \quad K \in \grid.
\end{align*}
\end{subequations}

Note that all finite volume schemes for the Euler systems defined above require the positivity of $\vr_h(t),$ $t>0.$

\section{Positivity of the discrete density and pressure}\label{S:positive}

As observed above, \emph{positivity} of the discrete density is necessary for the scheme to be properly defined. Starting from positive
initial density $\vr_h (0) > 0$, the semi--discrete scheme admits the unique solution defined on a maximal time interval $[0, T_{\rm max})$,
$T_{\rm max} > 0$. In general, $T_{\rm max}$ may even depend on $h$ and shrink to zero for $h \to 0$. In order to avoid
this difficulty, suitable {\it a priori} bounds that would guarantee $\vr_h(t)$ being bounded below away from zero must be established. This problem has been treated
for the relevant \emph{fully discrete} schemes by e.g., Perthame and Shu \cite{PerthShu}. 
Note that these results are always \emph{conditioned} by
a kind of CFL stability condition or other relevant restrictions. Seen from this perspective, the existence of an \emph{unconditional} result for the semi--discrete scheme seems to be out of reach both at the discrete level and for the limit Euler system. To eliminate this problem, we shall therefore impose positivity of $\vr_h$ as our principal working
hypothesis:
\begin{equation} \label{pHYP}
\vr_h (t) \geq \underline{\vr} > 0 \ \mbox{uniformly for}\ t \in [0,T],\ h \to 0
\end{equation}
for a positive constant $\underline{\vr}$.

Positivity of the density at the \emph{discrete level}, meaning with the lower bound $\underline{\vr}_h$ depending on the step $h$, can be achieved by adding lower order
``damping'' terms to the right--hand side of the momentum equation \eqref{mom_cE1} and the energy equation \eqref{en_cE1},  namely,
\[
- h^\alpha \frac{\vc{m}_h(t)}{\vr_h(t)} \ \mbox{and} \ - h^\alpha \left| \frac{\vc{m}_h(t)}{\vr_h(t)} \right|^2.
\]
Indeed adding these terms would:
\begin{itemize}
\item
leave the entropy balance in the same form;
\item
produce a uniform upper-bound on the discrete velocity
\begin{equation} \label{vHyp}
\vc{u}_h(t) \equiv \frac{\vc{m}_h(t)}{\vr_h(t)}, \ \mbox{specifically}\ \vc{u}_h \in L^2(0,T; L^\infty(\Omega; R^N)),
\end{equation}
resulting from boundedness of the discrete total energy $E_h(t)$.
\end{itemize}

In the next section, we show how positivity of the density can be obtained under the hypothesis (\ref{vHyp}).

\subsection{Conditional positivity of the density}

In this section, we show positivity of the density under the extra hypothesis on the approximate velocity,
\begin{equation} \label{vHYPP}
\vc{u}_h \equiv \frac{\vc{m}_h(t)}{\vr_h(t)} \in L^2(0,T; L^\infty(\Omega)).
\end{equation}
We restrict ourselves to the case of constant numerical viscosities.

Thus the first two equations of the numerical scheme for the Euler system  read,
\begin{subequations}\label{aux_scheme_bE}
\begin{align}
\Dt{\vr_K(t)}&+\diht{(\vr_h(t)\vu_h(t))}-\lambda  h\laph{\vr_h(t)}=0, \label{aux_cont_bE} \\
\Dt{(\vr_K(t)\vu_K(t))}&+\diht{\big(\vr_h(t)(\vu_h(t)\otimes\vu_h(t))+p_h(t)\mathbb{I}\big)}-\lambda  h\laph{(\vr_h(t)\vu_h(t))}=0, \label{aux_mom_bE}
\end{align}
\end{subequations}
equipped with the relevant initial conditions.

\begin{Lemma}\label{LemP}
Let $\vr_h(0)>0,$ and let a couple $(\vr_h(t),\vu_h(t)),$ $t>0,$ satisfy the discrete continuity equation \eqref{aux_cont_bE}, where
$\vu_h$ belongs to the class \eqref{vHYPP}.
\\
Then
\begin{align*}
 \dv{\vr}(t)  > \underline{\vr}_h >  0, \quad t \in [0,T],\  K \in \grid.
\end{align*}
\end{Lemma}
\begin{proof}

Let $\dv{\vr}(t)$ be such that $\dv{\vr}(t)\leq \vr_L(t)$ for all $L\in\grid.$
Equation \eqref{aux_cont_bE}  can be rewritten as
\begin{equation}\label{ce}
\begin{aligned}
\Dt{\dv{\vr}(t)} &
=-\sum_{s=1}^N \der{\vr_h}\avrgK{u_h^s}- \dv{\vr}\diht{\vu_h}-\sum_{s=1}^N\laphs{\vr}\left(\frac{h^2}{2}\der{u_h^s}-\lambda h\right).&
\end{aligned}
\end{equation}
By the definition of $\lambda$ and
the minimality of $\dv{\vr}(t)$
we can conclude that
\begin{align*}
-\der{\vr_h}\avrgK{u_h^s} &= - \frac{1}{2}\left[\derp{\vr_h}+\derm{\vr_h}\right]\avrgK{u_h^s}& \\
&\geq -\frac{\lambda}{2}\left[\derp{\vr_h}-\derm{\vr_h}\right]= -\frac{\lambda h}{2}\laphs{\vr_h},&\\
-\laphs{\vr_h} \left(\frac{h^2}{2}\der{u_h^s}-\lambda h+\frac{\lambda h}{2}\right)&=  - \frac{h}{4}\laphs{\vr_h}   \left(\dvpp{u}^s -\lambda\right)+&\\
&+\frac{h}{4}\laphs{\vr_h} \left(\dvmm{u}^s+\lambda\right) \geq 0 ,
\end{align*}
and consequently, equation  \eqref{ce} becomes
\begin{align*}
\Dt{\dv{\vr}(t)} &\geq  -\dv{\vr}\diht{\vu_h}.&
\end{align*}

As $\vc{u}_h$ satisfies (\ref{vHYPP}), we easily deduce a bound on the discrete divergence,
\[
{\diht{\vu_h}} \in L^2(0,T; L^\infty(\Omega)).
\]
Thus
the Gronwall inequality together with the assumption $\dv{\vr}(0)>0,$ $K\in\grid,$ finally yields for all $L\in\grid$ that
$\dvpp{\vr}(t) \geq \dv{\vr}(t) >0 ,$ $t\in [0,T].$
\end{proof}

Under the hypothesis (\ref{vHYPP}), setting $\vm_h\equiv \vr_h\vu_h$
and comparing \eqref{aux_cont_bE} with \eqref{cont_bE} or \eqref{cont_cE}, we realize that both formulations  are equivalent.
Analogous results hold for the schemes \eqref{scheme_bE1} and \eqref{scheme_cE1} with the local Lax-Friedrichs flux for both Euler systems, respectively.

\subsection{Positivity of discrete pressure}

Recall the entropy $\eta(\vU_h)=\vr_h \mathcal{S}_{\chi} (\vU_h),$ with $\mathcal{S}_\chi=\chi\circ S$ as in  Remark~\ref{R5}, is a concave function. The discrete entropy inequality \eqref{dis_en_ineq} holds,  cf.  \cite{Har83}, and may be used similarly to \cite{Tadmor86} for showing the minimal entropy principle. In particular, the relation between the initial density and temperature is time invariant and gives rise to the positivity of pressure.

\begin{Lemma}\label{LemDT}
Let the initial density and temperature for the complete Euler system  satisfy
\begin{align}\label{den_ini_cE}
0 < \dv{\vr}(0)\leq \Ov{C}(\dv{\vt}(0))^{\cv}, \ \Ov{C}>0, \ \mbox{ for all } K \in \grid,
\end{align}
where $\displaystyle \dv{\vt}(0) = \frac{(\gamma-1)}{\dv{\vr}(0)}\left(\dv{E}(0)-\frac{1}{2}\frac{|\dv{\vm}(0)|^2}{\dv{\vr}(0)}\right).$ \\
Then, for all $K\in\grid,$ it holds that
\begin{align}\label{den_temp}
0<\dv{\vr}(t)\leq \Ov{C}(\dv{\vt}(t))^{\cv},
\quad t \in [0,T],
\end{align}
where $ \displaystyle \dv{\vt}(t) = \frac{(\gamma-1)}{\dv{\vr}(t)}\left(\dv{E}(t)-\frac{1}{2}\frac{|\dv{\vm}(t)|^2}{\dv{\vr}(t)}\right).$
In particular, $\dv{p}(t)=\dv{\vr}(t)\dv{\vt}(t) >0,$ $t \in [0,T].$
\end{Lemma}
\begin{proof}
Recall that the renormalized entropy in our case, cf. \eqref{pressure_cE} and \eqref{S_cE}, can be rewritten as
\begin{align*}
\eta &=\vr \mathcal{S}_{\chi}=\vr\chi\left( \log \left(
\frac{(\gamma-1)}{\vr^{\gamma}}\left(E-\frac{1}{2}\frac{|\vm|^2}{\vr}\right)
\right)\right).
\end{align*}
Following \cite{BF} we now take the function $\chi$ satisfying \eqref{S_chi} to be such that
 \begin{align}\label{chi}
\chi'(z) \geq 0, \quad \chi(z)=\left\{\begin{array}{ll}
 <0, & z < z_0 \\
 0, & z \geq z_0, \\
\end{array}  \right., \quad z_0=(\gamma-1)\ln(1/\Ov{C}).
 \end{align}
Under the assumption  \eqref{den_ini_cE} it holds that
$$  \log \left(
\frac{(\gamma-1)}{\dv{\vr}(0)^{\gamma}}\left(\dv{E}(0)-\frac{1}{2}\frac{|\dv{\vm}(0)|^2}{\dv{\vr}(0)}\right)
\right)=\log\left(\frac{(\dv{\vt}(0))^{\cv}}{\dv{\vr}(0)}\right)\geq z_0,$$
which combined with   \eqref{chi} implies $\eta(\dv{\vU}(0))= 0.$
Thus, the sum of the discrete entropy inequality \eqref{dis_en_ineq}  integrated in time yields
\begin{align}\label{aux1}
\sum_{K\in\grid}{\eta(\dv{\vU}(t))}\geq  \sum_{K\in\grid}{\eta(\dv{\vU}(0))}=0, \quad  t \in [0,T].
\end{align}
From inequality \eqref{aux1} it directly follows that
\begin{align*}
\sum_{K\in\grid} \dv{\vr}(t) \chi\left(
\log \left(
\frac{(\gamma-1)}{\dv{\vr}(t)^{\gamma}}\left(\dv{E}(t)-\frac{1}{2}\frac{|\dv{\vm}(t)|^2}{\dv{\vr}(t)}\right)
\right)
\right)=
\sum_{K\in\grid} \dv{\vr}(t) \chi\left( \log\left(\frac{(\dv{\vt}(t))^{\cv}}{\dv{\vr}(t)}\right)\right) \geq 0.
\end{align*}
Consequently, employing  \eqref{chi} and the positivity of $\dv{\vr}(t),$ we get that $$\log \left(
\frac{(\gamma-1)}{\dv{\vr}(t)^{\gamma}}\left(\dv{E}(t)-\frac{1}{2}\frac{|\dv{\vm}(t)|^2}{\dv{\vr}(t)}\right)
\right)=\log\left(\frac{(\dv{\vt}(t))^{\cv}}{\dv{\vr}(t)}\right) \geq z_0,\quad t \in [0,T],$$ which concludes the proof.
\end{proof}

\begin{Lemma}\label{L}
Let $\vU_h=[\vr_h,\vm_h,E_h]$ be a  solution of the complete Euler system constructed via the numerical schemes \eqref{scheme_cE1} or \eqref{scheme_cE}. In addition, suppose that
\begin{align*}
0 < \Un{\vr} \leq \vr_h(t), \ E_h(t) \leq \Ov{E}  \ \mbox{uniformly for } h\rightarrow 0,  \ t \in [0,T]
\end{align*}
for some constants $\Un{\vr},$ $\Ov{E}.$ \\
Then there exist constants $\Ov{\vr},$ $\Un{\vt},$ $ \Ov{\vt},$ $\Un{p},$ $\Ov{p},$ $\Ov{\vm}$ such that
\begin{align}\label{nerovnosti}
\vr_h \leq \Ov{\vr}(t), \ |\vm_h(t)| \leq \Ov{\vm},\ \ 0 < \Un{\vt} \leq \vt_h(t) \leq \Ov{\vt},  \ 0 < \Un{p} \leq p_h(t) \leq \Ov{p} \ \mbox{ uniformly for } h \rightarrow 0, \ t \in [0,T].
\end{align}
\end{Lemma}

 \begin{proof}
Since we already know that the pressure $p_h$ is positive, we have
\begin{align*}
0 < p_h=(\gamma-1)\left( E_h - \frac{1}{2}\frac{|\vm_h|^2}{\vr_h}\right) \leq \Ov{E},
\end{align*}
which yields the existence of $\Ov{p}$  satisfying \eqref{nerovnosti}.
From Lemma~\ref{LemDT} we also have
\begin{align*}
0<\vr_h\leq \Ov{C}(\vt_h)^{\cv}.
\end{align*}
Therefore,
\begin{align*}
0 < \Un{\vr}^\gamma \leq \vr_h^{\gamma} \leq \Ov{C}^{\gamma-1}\vr_h\vt_h =  \Ov{C}^{\gamma-1} p_h \leq  \Ov{C}^{\gamma-1} \Ov{E},
\end{align*}
which gives the existence of $\Ov{\vr},$ $\Un{p},$ $\Un{\vt},$ $\Ov{\vt}.$ Finally,
\begin{align*}
|\vm_h|^2 \leq 2\vr_h E_h \leq 2\Ov{\vr}\Ov{E}.
\end{align*}

\end{proof}

\section{Stability of numerical schemes}

We show the stability of the numerical schemes defined in Section~\ref{S:scheme} by deriving   a priori estimates.

\subsection{A priori estimates for the barotropic Euler system}

Firstly, we sum up the continuity equation \eqref{cont_bE} (or \eqref{cont_bE1}) multiplied by $h^N$ for all $K\in \grid$ and integrate in time to get
\begin{align*}
\intO{\vr_h(t)} = \intO{\vr_h(0)}.
\end{align*}
The positivity of $\vr_h(t)$ then indicates  $\vr_h \in L^{\infty}(0,T;L^1(\Omega)).$
Further we know that our entropy stable finite volume scheme \eqref{scheme_bE} directly yields the  discrete entropy inequality.
 It is important to point out that for barotropic flow the energy plays the role of entropy (with a negative sign). Denoting
\begin{align}\label{entropy_bE}
\eta(\dv{\vU})= \frac{1}{2}\frac{|\dv{\vm}|^2}{\dv{\vr}} + P(\dv{\vr}),
\end{align}
we obtain for the entropy stable finite volume schemes the discrete energy inequality
\begin{align}\label{dis_en_bE}
\Dt{}\dv{\eta(\vU}(t))+\dih{\bfQ_h(t)} \leq 0, \ K \in \grid.
\end{align}
Since the numerical entropy flux  given by \eqref{def_Q} is conservative, i.e., $\displaystyle\sum_{K\in\grid}\dih{\bfQ_h}=0,$  the integral of \eqref{dis_en_bE}  yields
\begin{align*}
\intO{\eta(\vU_h(t))} \leq \intO{\eta(\vU_h(0))}.
\end{align*}
Similarly as above, the latter  inequality gives rise to $\eta(\vU_h)\in L^{\infty}(0,T;L^1(\Omega)).$ Noting also  \eqref{pressure_bE} and \eqref{ppot},  we conclude the a priori estimates for the barotropic Euler equations:
\begin{equation}\label{apriori_bE}
\begin{aligned}
\vr_h \in   L^{\infty}(0,T;L^{\gamma}(\Omega)),\ \gamma >1,  \qquad
p_h \in L^{\infty}(0,T;L^1(\Omega)), \\
  \sqrt{\vr_h}\vu_h \in L^{\infty}(0,T;L^2(\Omega)),\ \mbox{ and } \ \vm_h=\vr_h \vu_h \in L^{\infty}(0,T;L^{r}(\Omega)),\ r=\frac{2\gamma}{1+\gamma} >1 .
\end{aligned}
\end{equation}

\subsection{A priori estimates for the complete Euler system}

We sum up  equation of continuity \eqref{cont_cE} (or \eqref{cont_cE1}) and energy equation \eqref{en_cE} (or \eqref{en_cE1}) multiplied by $h^N$ over $K\in\grid.$ Due to the periodic boundary conditions we get
\begin{align}\label{coonservation_cE}
&\intO{\vr_h(t)} = \intO{\vr_h(0)}, \quad \intO{E_h(t)} = \intO{E_h(0)}.&
\end{align}
In Section~\ref{S:positive} we have shown that $\vr_h(t),$ $p_h(t) >0,$ and thus also $E_h(t) >0$ for $t\in[0,T].$
The conservation of mass and energy \eqref{coonservation_cE} combined with  \eqref{den_temp}  imply the a priori estimates for the complete Euler system. Namely,
\begin{equation}
\begin{aligned}\label{apriori_cE}
\vr_h \in  L^{\infty}(0,T;L^{\gamma}(\Omega)), \ \gamma >1, \ p_h \in L^{\infty}(0,T;L^1(\Omega)), \ E_h \in L^{\infty}(0,T;L^1(\Omega)) \\
  \sqrt{\vr_h}\vu_h \in L^{\infty}(0,T;L^2(\Omega)),\ \mbox{ and } \ \vm_h=\vr_h \vu_h \in L^{\infty}(0,T;L^{r}(\Omega)),  \ r=\frac{2\gamma}{1+\gamma} >1 .
\end{aligned}
\end{equation}

\section{Consistency}

In this section our aim is to show consistency of the entropy stable finite volume schemes \eqref{scheme_bE1}, \eqref{scheme_bE}
and \eqref{scheme_cE1}, \eqref{scheme_cE}.  We derive suitable formulations of the continuity and momentum equations that are the same for the barotropic and the complete Euler systems. In addition, for the complete Euler system, we also show consistency of the entropy inequality.

\subsection{Consistency formulation of continuity and momentum equations}

Let us multiply the continuity equations \eqref{cont_bE1} or \eqref{cont_bE} (for the barotropic Euler) and \eqref{cont_cE1} or \eqref{cont_cE} (for the complete Euler) by $h^N\dv{(\Pi_h\varphi(t))},$ with $\varphi \in C^3([0,T)\times\Omega),$ and the momentum equations \eqref{mom_bE1} or \eqref{mom_bE} (for the barotropic Euler)  and \eqref{mom_cE1} or \eqref{mom_cE}  (for the complete Euler) by $h^N\dv{(\Pi_h\boldsymbol{\varphi}(t))},$ with  $\boldsymbol{\varphi} \in C^3([0,T)\times\Omega;R^N).$
We  sum the resulting equations over $K\in\grid$ and integrate in time.
 The a priori estimates \eqref{apriori_bE} or \eqref{apriori_cE} for both the barotropic and the complete Euler systems combined with some boundedness assumptions specified below shall allow us to show the consistency.

\subsubsection*{Time derivative}

Integration by parts with respect to time leads to
\begin{align*}
h^N\intOT{\Dt\sum_{K\in\grid} {\dv{\vr}(t)}\dv{(\Pi_h\varphi(t))}}&=\intOT{\Dt{}\intO{\dv{\vr}(t)\varphi(t,x)}}& \\
&=\left[\intO{\vr_h(\tau)\varphi(\tau,\cdot)}\right]_{\tau=0}^{\tau=T}- \intOT{\intO{\vr_h(t) \pd{\varphi(t,x)}{t} }}
\end{align*}
in the continuity equations,
and similarly to
\begin{align*}
&h^N\intOT{\Dt{}\sum_{K\in\grid} \dv{\vm}(t)\cdot\dv{(\Pi_h\boldsymbol{\varphi}(t))}}& \\
&=\left[\intO{\vm_h(\tau)\cdot \boldsymbol{\varphi}(\tau,x)}\right]_{\tau=0}^{\tau=T}- \intOT{\intOh{\vm_h(t) \cdot\pd{\boldsymbol{\varphi}(t,x)}{t} }}
\end{align*}
in the momentum equations.

\subsubsection*{Convective terms}

To treat the convective terms in the continuity equations we use the discrete integration by parts and the Taylor expansion to get
\begin{align*}
&h^N\intOT{\sum_{K\in\grid} \diht{\vm_h(t)}\dv{(\Pi_h\varphi(t))}}& \\
&=-h^N\intOT{\sum_{K\in\grid} \sum_{s=1}^N m_K^s(t)\left(\intOhi{\frac{\varphi(t,x+h\bfe_s) -\varphi(t,x-h\bfe_s)}{2h}}\right)}&\\
&=-\intOT{\intO{\vm_h(t)\cdot \grad\varphi(t,x)}}+r_1,&
\end{align*}
where
term $r_1$ is estimated as follows
\begin{align}\label{r1}
r_1\lesssim h \n{C(0,T)}{\DDt{2}{\varphi}{x}(\hat{x})}\n{L^{\infty}(L^1)}{\vm_h}, \quad \mbox{where} \quad
\DDt{2}{\varphi}{x}\defeq \left(\frac{\partial^2\varphi}{\partial x_i \partial x_j}\right)_{i,j=1}^N .
\end{align}
Point $\hat{x}$  appears in the  remainder of the Taylor expansion and lies either between the points $x+h\bfe_s$ and $x$ or the points $x$ and $x-h\bfe_s.$
\medskip

We proceed analogously with the convective term in the momentum equations, i.e.,
\begin{align*}
&h^N\intOT{\sum_{K\in\grid} \diht{\bigg(\frac{\vm_h(t)\otimes\vm_h(t)}{\vr_h(t)}+p_h(t)\mathbb{I}\bigg)}\dv{(\Pi_h\boldsymbol{\varphi(t)})}}& \\
&=-h^N\intOT{\sum_{K\in\grid} \sum_{s=1}^N\sum_{z=1}^N\bigg(\frac{m_h^s(t)m_h^z(t)}{\vr_h(t)}+p_h(t)\bigg)\left(\intOhi{\frac{\varphi^z(t,x+h\bfe_s) -\varphi^z(t,x-h\bfe_s)}{2h}}\right)}&\\
&=-\intOT{\intO{\bigg(\frac{\vm_h(t)\otimes\vm_h(t)}{\vr_h(t)}+p_h(t)\mathbb{I}\bigg)\cdot\grad\boldsymbol{\varphi}(t,x)}}+r_2,&
\end{align*}
where
term $r_2$ is bounded by
\begin{align*}
r_2
\lesssim h \n{C(0,T)}{\DDt{2}{\boldsymbol{\varphi}}{x}(\hat{x})}\Bigg\{\n{L^{\infty}(L^2)}{\sqrt{\vr_h(t)}\vu_h(t)}+\n{L^{\infty}(L^1)}{p_h(t)}\Bigg\}.
\end{align*}

\subsubsection*{Numerical diffusion}

 Diffusive terms of the numerical schemes (\ref{scheme_bE1}), (\ref{scheme_bE}) and (\ref{scheme_cE1}), (\ref{scheme_cE})  will be computed separately for the  global and the local numerical diffusion coefficients $\lambda$ and $\bv{\lambda}$, respectively. For the global numerical diffusion coefficient we can write
\begin{align*}
& h^{N+1}\intOT{ \lambda\sum_{K\in\grid} \laph{\vU_h(t)}\dv{(\Pi_h\boldsymbol{\varphi}(t))}}& \nonumber \\ \nonumber
& = h^{N+1}\intOT{\lambda\sum_{K\in\grid} \dv{\vU}(t)\left(\intOhi{\sum_{s=1}^N\frac{\boldsymbol{\varphi}(x+h\bfe_s) -2\boldsymbol{\varphi}(x) +\boldsymbol{\varphi}(x-h\bfe_s)}{h^2}}\right)}& \\ 
&= h^N\intOT{\lambda\intO{\vU_h(t)\Delta_x{\boldsymbol{\varphi}(t,x)}}} + r_3.
\end{align*}
Similarly as in \eqref{r1} the remainders of the Taylor expansions result in term $r_3$ that  is bounded by
\begin{align*}
r_3 \lesssim   h   \n{C(0,T)}{\DDt{3}{\boldsymbol{\varphi}}{x}(\tilde{x})}\n{L^{\infty}(L^{1})}{\vU_h}\intOT{\lambda}.
\end{align*}
Moreover, the term stemming from the numerical diffusion is of order $\mathcal{O}(h).$ Indeed, we have
\begin{align*}
 h\intOT{\lambda\intO{\vU_h(t)\Delta_x{\boldsymbol{\varphi}(t,x)}}}&
 \leq  h T  \n{\infty}{\Delta_x{\boldsymbol{\varphi}}}\n{L^{\infty}(L^{1})}{\vU_h}\intOT{\lambda} .
\end{align*}
Assuming a finite speed of waves propagation, i.e., there  exists $\Ov{\lambda}>0$ such that $\lambda \leq \Ov{\lambda},$ the latter term goes to 0 as $h \rightarrow 0.$

\medskip
For the local  numerical diffusion coefficient we are able to prove consistency of the numerical diffusion term without the assumption on the finite speed of propagation. Indeed, considering the diffusion terms we obtain
\begin{align} \label{num_dif}
& h^{N-1}\intOT{ \sum_{K\in\grid} \sum_{\sigma \in \partial K}  \lambda_{\sigma} \jump{\vU_h(t)} ( \bfn_K^+\cdot\bfe_s)  \left(\Pi_h\boldsymbol{\varphi}(t)\right)_K
}.
\end{align}
The terms belonging to an arbitrary but fixed face $\sigma=K|L$  are
\begin{align}\label{sum_jump}
&\frac{1}{h} \intOT{   \left( \bv{\lambda} \jump{\vU_h(t)} \int_K \boldsymbol{\varphi}(t) \dx -
\bv{\lambda} \jump{\vU_h(t)} \int_L \boldsymbol{\varphi}(t)\dx \right)}.
\end{align}
Let us now consider an arbitrary but fixed point $\tilde x \in \sigma$; w.l.o.g. let $\tilde x=(\tilde x_s,  x')$, $x' \in \Bbb R^{N-1}$, $s=1,\dots, N.$
The Taylor expansion for $x = (x_s, x') \in K$ with respect to $(\tilde x_s,  x')$ gives
$$
\boldsymbol{\varphi} (x_s, x')= \boldsymbol{\varphi} (\tilde x_s, x') - \xi \partial_s \boldsymbol{\varphi} (\tilde x_s, x') + {\cal O}(h^2),
$$
where $\xi \in (0,h).$ Analogously, we have for $x = (\tilde x_s, x') \in L$
$$
\boldsymbol{\varphi} (x_s, x')= \boldsymbol{\varphi} (\tilde x_s, x') + \xi \partial_s \boldsymbol{\varphi} (\tilde x_s, x') + {\cal O}(h^2).
$$
Substituting the above Taylor expansions in (\ref{sum_jump}) we directly see that the terms multiplied by $\boldsymbol{\varphi} (\tilde x_s, x')$ vanish. The resulting terms give
\begin{eqnarray*}
&&\left|\int_0^T - \frac{1}{h} \bv{\lambda} \jump{\vU_h} \int_0^h \int_\sigma \xi \partial_s{\boldsymbol{\varphi}} (\tilde x_s, x') {\rm d} \xi {\rm d} S_{x'}
+ \frac{1}{h} \bv{\lambda} \jump{\vU_h} \int_0^h \int_\sigma  - \xi \partial_s{\boldsymbol{\varphi}} (\tilde x_s, x') {\rm d} \xi {\rm d}S_{x'} \dt\right|\\
&& \leq
 \frac{2}{h} \int_0^T \Big |  \bv{\lambda} \jump{\vU_h} \int_0^h \int_\sigma \xi \partial_s{\boldsymbol{\varphi}} (\tilde x_s, x') {\rm d} \xi {\rm d} S_{x'} \Big | \dt\\
&&  \lesssim h^N \int_0^T \bv{\lambda} \Big| \jump{\vU_h} \Big| dt \, \| \boldsymbol{\varphi}\|_{C^1([0,T]\times\Omega)} \rightarrow 0 \qquad \mbox{ for }\  h \to 0.
\end{eqnarray*}
The last convergence follows from the weak BV property (\ref{weakBV1}) and  implies the consistency of the numerical diffusion term (\ref{num_dif}).

\begin{Remark}[weak BV (\ref{weakBV1}) holds for the finite volume schemes (\ref{scheme_bE1}), (\ref{scheme_cE1})]\label{R:weakBV}\
\\
In what follows we show that the finite volume schemes (\ref{scheme_bE1}) and (\ref{scheme_cE1})
with the local numerical diffusion satisfy the weak BV estimate (\ref{weakBV1}).
To unify the argumentation  we set in this remark $\eta := - \vr\mathcal{S}_{\chi}$  for the complete Euler equations in order to work with the  convex entropy for both barotropic and complete Euler systems.
Let us assume  that
 \begin{itemize}
 \item no vacuum appears, i.e.~
 \begin{equation}\label{positive_rho}
 \exists \, \Un{\vr} > 0: \vr_h(t) \geq  \Un{\vr}
 \end{equation}
 \item entropy Hessian ist strictly positive definite, i.e.~
 \begin{equation}\label{positive_e}
 \exists \, \Un{\eta} > 0:  \frac{d^2\eta(\vU)}{d\vU^2} \geq \Un{\eta} \mathbb{I}, \quad \mathbb{I} \mbox{ is a unit matrix}.
 \end{equation}
 \end{itemize}
The entropy residual {\cg $ r_\sigma$} arising in the discrete entropy inequality, that is obtained by multiplying the conservation law (\ref{conservative_system}) by
$\nabla_{\vU} \eta(\vU)$, reads, see, e.g., \cite{Tad03}, \cite{fjdiss},
\begin{equation*}
r_{\sigma} = - \delta_\sigma \, \jump{\vU} \, \jump{\vV}.
\end{equation*}
Here $\delta_\sigma > \bv{\lambda} / 2 > 0.$ For the Euler equations it holds that $\bv{\lambda} = \max(|\vu_K | + c_K, |\vu_L | + c_L)$,
$\sigma = K|L$. Furthermore,  we have for the barotropic and the complete Euler equations
$c = \sqrt{\gamma \vr^{\gamma-1 }}$ and $c = \sqrt{\gamma p/ \vr}$, respectively.

It follows from the construction of the entropy stable schemes that the entropy residual is negative, see \cite{Tad87,Tad03,fjdiss}. Consequently,
integrating the discrete entropy inequality over $\Omega$ and over time interval $(0,T)$ yields
$$
\int_{\Omega} \eta(\vU_h(T)) \dx - \int_0^T \sum_{\sigma \in \mathcal{E}} h^{N-1} r_\sigma \, \dt \leq \int_{\Omega} \eta(\vU_h(0)) \dx
\leq \mbox{const.}
$$
{\cg Furthermore, it holds that $\eta(\vU_h(t)) \geq \tilde \eta, $  \ $t\in (0,T)$. Indeed, for the barotropic Euler system this bound holds  due to (\ref{positive_rho}) and thus (\ref{eq1}) follows. For the complete Euler system it holds for
any $\eta = -\vr S_\chi$  since $\chi$ is bounded from above and (\ref{positive_rho}) holds. Thus, passing to the limit with $\chi(Z) \to Z $
in the entropy inequality we obtain finally}
\begin{equation}\label{eq1}
- \int_0^T \sum_{\sigma \in \mathcal{E}} h^{N} r_\sigma  \, \dt \to 0   \qquad \mbox{for }\  h \to 0.
\end{equation}
Assumption (\ref{positive_e}) and the mean value theorem imply
\begin{equation*}
\jump{\vU} = \vU'(\tilde\vV) \jump{\vV} = \left(\frac{d^2(\eta(\tilde\vU))}{d\vU^2}\right)^{-1} \jump{\vV}
\end{equation*}
and thus
\begin{equation*} 
\Un{\eta}\jump{\vU_h} \leq \jump{\vV_h}.
\end{equation*}
Consequently, we have
\begin{equation} \label{eq3}
\frac{\Un{\eta}}{2} \int_0^T \sum_{\sigma \in \mathcal{E}} h^{N} \bv{\lambda} \jump{\vU}^2 \,\dt \leq
\int_0^T \sum_{\sigma \in \mathcal{E}} h^{N} \delta_\sigma \jump{\vU} \jump{\vV} \, \dt,
\end{equation}
where the last term tends to 0 for $h \to 0$ according to (\ref{eq1}).
It remains to show that the weak BV estimate (\ref{weakBV1}) holds. Indeed,
\begin{equation}\label{eq4}
\int_0^T \sum_{\sigma \in \mathcal{E}} h^{N} \bv{\lambda} |\jump{\vU}| \dt \leq
\left( \int _0^T \sum_{\sigma \in \mathcal{E}} h^{N} \bv{\lambda} \dt \right)^{1/2}
\left( \int_0^T \sum_{\sigma \in \mathcal{E}} h^{N} \bv{\lambda} |\jump{\vU}|^2 \dt \right)^{1/2}.
\end{equation}
The second term on the RHS of (\ref{eq4}) tends to 0 due to (\ref{eq3}) and (\ref{eq1}). To show the boundedness of the first term we apply the discrete trace inequality that holds for arbitrary piecewise constant function $f_h$, cf., e.g., \cite{FL17}
$$
\| f_h \|_{L^p(\partial K)} \leq h^{-1/p} \| f_h \|_{L^p(K)}, \qquad 1 \leq p \leq \infty.
$$
Thus,
\begin{eqnarray*}
&&\int _0^T \sum_{\sigma \in \mathcal{E}_{in}} h^{N} \bv{\lambda}  \dt \leq h\int_0^T \sum_{K \in \grid}
\sum_{\sigma \in \partial K} \int_\sigma \bv{\lambda} \rm{d}S \dt
\\
&& \lesssim h  \int _0^T \sum_{K \in \grid} \frac{1}{h}  \int_K | \lambda(U_K) | \dx \dt \leq \mbox{const.}, \quad  \lambda(U_K) = |\vu_K| + c_K.
\end{eqnarray*}
The last inequality follows from the assumption (\ref{positive_rho}) and from  a priori estimates (\ref{apriori_bE}) and (\ref{apriori_cE})
for the barotropic and the complete Euler equations, respectively.
In  conclusion, the weak BV estimate (\ref{weakBV1}) holds for the finite volume schemes (\ref{scheme_bE1}) and (\ref{scheme_cE1}) provided
there is no vacuum and the entropy Hessian is strictly positive definite for barotropic and strictly negative definite for the complete Euler equations, respectively.
\end{Remark}

\subsection{Consistency formulation of the entropy inequality for the complete Euler system}

For the complete Euler system we shall also derive a suitable consistency formulation of the  discrete  entropy inequality \eqref{dis_en_ineq} for
\begin{align}\label{entropy_cE}
\eta(\vU_h)=\vr_h\chi\left(\frac{1}{\gamma-1}\log\left((\gamma - 1) \frac{E_h - \frac{1}{2} \frac{|\vm_h|^2}{\vr_h} }{\vr_h^\gamma}\right)\right).
\end{align}
 Due to a priori estimates \eqref{apriori_cE}, Lemma~\ref{LemDT} and assumptions \eqref{S_chi} on $\chi$ we know that $\eta(\vU_h) \in L^{\infty}(0,T;L^{\gamma}(\Omega)).$ By the same token we know that
 \begin{align}\label{entropy_flux_cE}
 \bfq(\vU_h)=\vm_h \chi\left(\frac{1}{\gamma-1}\log\left((\gamma - 1) \frac{E_h - \frac{1}{2} \frac{|\vm_h|^2}{\vr_h} }{\vr_h^\gamma}\right)\right) \in L^{\infty}(0,T;L^{r}(\Omega)), \ r=\frac{2\gamma}{\gamma+1}.
 \end{align}
In what follows we assume
that the numerical entropy flux $\bfQ_h$ is \emph{globally Lipschitz-continuous}, i.e., there exists a $\tilde{C}>0$ such that for any $\sigma=K|L$ it holds that
\begin{align}\label{globLC}
\|\bv{\bfQ}(t)-\bfq(\vU_K(t))\|\equiv\|\bfQ_h(\vU_K(t),\vU_L(t))-\bfq(\vU_K(t))\|\leq \tilde{C}\|\vU_K(t)-\vU_L(t)\|,\quad L=K+h\bfe_s.
\end{align}
To derive the consistency formulation of the discrete renormalized entropy inequality we multiple  \eqref{dis_en_ineq} by $h^N\dv{(\Pi_h\varphi(t))},$ for any $\varphi \in C^{2}([0,T)\times\Omega),$ $\varphi \geq 0,$ and integrate in time to get:
\begin{itemize}
\item[] {\bf Time derivative:}
\begin{align*}
&h^N\intOT{\sum_{K\in\grid} \Dt{\eta(\dv{\vU}(t))}\dv{(\Pi_h\varphi(t))}}& \\
&=\left[\intO{\eta(\dv{\vU}(\tau))\varphi(\tau,\cdot)}\right]_{\tau=0}^{\tau=T}- \intOT{\intO{\eta(\dv{\vU}(t)) \pd{\varphi(t,x)}{t} }}.
\end{align*}

\item[] {\bf Convective term:} discrete integration by parts yields
\begin{align*}
&h^N\intOT{\sum_{K\in\grid}\dih{\bfQ_h(t)}\dv{(\Pi_h\varphi(t))}}& \\
& = -h^N\intOT{ \sum_{s=1}^N\sum_{\sigma\in\edge} \bv{Q}^s(t)\derb{(\Pi_h\varphi(t))}} = & \\
&=-h^N\intOT{ \sum_{s=1}^N \sum_{\sigma\in\edge}\big(\bv{Q}(t)-q^s(\dv{\vU}(t))\big)\derp{(\Pi_h\varphi(t))}}\ -& \\
&-\intOT{\intO{  \bfq(\dv{\vU}(t))\cdot\nabla_x{\varphi}(t,x)}}+R,
\end{align*}
where the last two terms with
\begin{align*}
R\lesssim h\n{C(0,T)}{\nabla_x\varphi(\hat{x})}\n{L^{\infty}(L^{r})}{\bfq(\vU_h)}
\end{align*}
appeared as a result of the identity
\begin{align*}
&h^N\derp{(\Pi_h\varphi(t))} = \intOhi{\frac{\varphi(t,x+h\bfe_s) -\varphi(t,x)}{h}}=\intOhi{\nabla_x\varphi(t,x)-\frac{h}{2}\DDt{2}{\varphi(\hat{x})}{x}}.
\end{align*}
What remains is to show that
\begin{align*}
-h^N\intOT{ \sum_{s=1}^N\sum_{\sigma\in\edge} \big(\bv{Q}(t)-q^s(\dv{\vU}(t))\big)\derp{(\Pi_h\varphi(t))}}
 =\mathcal{O}(h).
\end{align*}
Due to the  global Lipschitz continuity of $\bfQ_h,$ cf. \eqref{globLC}, we get     the following inequality
\begin{equation}\label{ineq}
\begin{aligned}
&-h^N\intOT{ \sum_{s=1}^N\sum_{\sigma\in\edge} \big(\bv{Q}(t)-q^s(\dv{\vU}(t))\big)\derp{(\Pi_h\varphi(t))}} & \\
&\leq C_L h^N \intOT{ \sum_{K\in\grid}\left[ \big\|\vU_K(t)-\vU_L(t)\big\|\sum_{s=1}^N\big|\derp{(\Pi_h\varphi(t))}\big|\right]}& \\
& \leq C_L  \left(h^N\intOT{\sum_{K\in\grid} \big\|\vU_K(t)-\vU_L(t)\big\|^2}\right)^{1/2}\left( \intOT{\intO{\sum_{s=1}^N\left|\DDt{}{\varphi}{x_s}(x)-\frac{h}{2}\DDt{2}{\varphi(\tilde{x})}{x_s}\right|^{2}}}\right)^{1/2}\\
&\lesssim C_L  \left(\intOT{\sum_{\sigma\in\edge} \left|\jump{\vU_h(t)}\right|^2 h^N}\right)^{1/2} \Bigg\{\|\nabla_x\varphi\|_{\infty}+h\left\|\DDt{2}{\varphi(\tilde{x})}{x}\right\|_{C(0,T)}\Bigg\}.
\end{aligned}
\end{equation}
To show that the first term  in \eqref{ineq} goes to zero, we follow analogous arguments as in Remark~\ref{R:weakBV}.
{\cg We assume strict positivity of the density \eqref{positive_rho}. Furthermore, for  physical entropy we assume
its uniform concavity, i.e.~strict positive definiteness of the Hessian for mathematical entropy, cf.~\eqref{positive_e}.
Applying Lemma~\ref{LemDT}}
we obtain from the control of the entropy residual $\bv{r}$ that there exists $\Un{\lambda} > 0$, such that
\begin{equation*}
\frac{\Un{\lambda}\Un{\eta}}{2} \int_0^T  \sum_{\sigma \in \mathcal{E}} h^{N}  \jump{\vU_h(t)}^2 \,\dt \leq
\int_0^T \sum_{\sigma \in \mathcal{E}} h^{N} \delta_\sigma \jump{\vU_h(t)} \jump{\vV_h(t)} \, \dt \rightarrow 0.
\end{equation*}
Finally, we have shown
\begin{align*}
&-h^N\intOT{ \sum_{s=1}^N\sum_{\sigma\in\edge} \big(\bv{Q}(t)-q^s(\dv{\vU}(t))\big)\derp{(\Pi_h\varphi(t))}}  \to 0 \mbox { as } h \to 0.
\end{align*}
\end{itemize}

\noindent Let us summarize the consistency results derived in this section.

\subsubsection*{Consistency formulation for the barotropic Euler system}

The consistency formulation of the numerical schemes \eqref{scheme_bE1} and \eqref{scheme_bE} for the barotropic Euler equations reads
\begin{equation}\label{consistency_bE}
\begin{aligned}
-\intO{\vr_h(0)\varphi(0,\cdot)}&= \intOT{\intO{\vr_h \pd{\varphi}{t} + \vm_h\cdot\grad{\varphi}}}+ \mathcal{O}(h)  \\
&\mbox { for any } \varphi \in C^3_c([0,T)\times\Omega); \\
-\intO{\vm_h(0)\cdot\boldsymbol{\varphi}(0,\cdot)}&= \intOT{\intO{\vm_h\cdot \pd{\boldsymbol{\varphi}}{t} }}+ &\\
&+\intOT{\intO{\bigg(\frac{\vm_h\otimes\vm_h}{\vr_h}+p_h\mathbb{I}\bigg)\cdot\grad\boldsymbol{\varphi} }}+ \mathcal{O}(h)\\
&\mbox { for any }  \boldsymbol{\varphi}  \in C^3_c([0,T)\times\Omega;R^N);\\
\left[\intO{\eta(\vU_h(t))}\right]_{t=0}^{t=\tau}&\leq 0,\ \mbox{ for a.a. } 0 \leq \tau \leq T \mbox{ with } \eta(\vU_h)=\frac{1}{2}\frac{|\vm_h|^2}{\vr_h}-P(\vr_h).
\end{aligned}
\end{equation}

\begin{Lemma}\label{L:AbE}
Let us assume that
\begin{itemize}
\item[(A1)] no vacuum appears, i.e., there exists $\Un{\vr} >0 $, such that $ \vr_h(t) \geq \Un{\vr}$, cf.~(\ref{pHYP})
\item[(A2)] {\cg if $1 < \gamma < 3$ then there exists $\overline{\vr} >0 $, such that $ \vr_h(t) \leq \overline{\vr}$.}
\end{itemize}
Then the local Lax-Friedrichs scheme (\ref{scheme_bE1}) is consistent with the barotropic Euler equations (\ref{M1}) and the consistency
formulation (\ref{consistency_bE}) holds. If we assume that
\begin{itemize}
{\cg \item[(A1)] no vacuum appears,  i.e., there exists $\Un{\vr} >0 $, such that $ \vr_h(t) \geq \Un{\vr}$}, cf.~(\ref{pHYP})
\item[(A3)] finite speed of propagation holds, i.e., there  exists $\Ov{\lambda} >0$, such that $\lambda(\vU_h(t)) \leq \Ov{\lambda}$ uniformly for $t \in [0,T]$ and $h\rightarrow 0,$
\end{itemize}
then the global Lax-Friedrichs scheme (\ref{scheme_bE}) is consistent with the barotropic Euler equations (\ref{M1}) and the consistency
formulation (\ref{consistency_bE}) holds.
\end{Lemma}

{\cg
\begin{proof}
The only  point to verify is to show that  (A1) and (A2) imply strict positive definiteness of the entropy Hessian. Indeed, we have for the barotropic Euler systems that
\begin{equation*}
\frac{d^2 \eta(\vU)}{d\vU^2} = \left( \begin{array}{c c}  a \gamma \vr^{\gamma-2} + \frac{| \vm |^2}{\vr^3} & -\frac{| \vm |}{\vr^2} \\
                                                             -\frac{| \vm |}{\vr^2} & \frac{1}{\vr}
                                                             \end{array}
                                                             \right).
\end{equation*}
Direct calculation yields the determinant and the trace of entropy Hessian, i.e. $\displaystyle \mbox{det} = a \gamma \vr^{\gamma -3}$ and
$\displaystyle \mbox{tr} = a \gamma \vr^{\gamma-2} + \frac{|\vm|^2}{\vr^3} + \frac{1}{\vr},$ respectively.  Consequently, for $\gamma \geq 3$
the Hessian is uniformly strictly positive if (A1) holds, for $1 < \gamma < 3 $ we need to require (A1) and (A2).
\end{proof}
}

\subsubsection*{Consistency formulation for the complete Euler system}

The consistency formulation of the numerical schemes \eqref{scheme_cE1} and \eqref{scheme_cE} for the complete Euler equations reads
\begin{equation}\label{consistency_cE}
\begin{aligned}
-\intO{\vr_h(0)\varphi(0,\cdot)}&= \intOT{\intO{\vr_h(t) \pd{\varphi(t,x)}{t} + \vm_h(t)\cdot\grad{\varphi(t,x)}}}+ \mathcal{O}(h) \\
&\mbox { for any } \varphi \in C^3_c([0,T)\times\Omega); \\
-\intO{\vm_h(0)\cdot\boldsymbol{\varphi}(0,\cdot)}&= \intOT{\intO{\vm_h(t)\cdot \pd{\boldsymbol{\varphi}(t,x)}{t} }}+ &\\
&+\intOT{\intO{\bigg(\frac{\vm_h(t)\otimes\vm_h(t)}{\vr_h(t)}+p_h(t)\mathbb{I}\bigg)\cdot\grad\boldsymbol{\varphi}(t,x) }}+ \mathcal{O}(h) \\
&\mbox { for any }  \boldsymbol{\varphi}  \in C^3_c([0,T)\times\Omega;R^N);\\
\left[\intO{E_h(t)}\right]_{t=0}^{t=\tau}&=0, \ \mbox{ for a.a. } 0 \leq \tau \leq T ;\\
-\intO{\eta(\vU_h(0))\varphi(0,\cdot)}&\geq \intOT{\intO{\eta(\vU_h(t))\cdot \pd{\varphi(t,x)}{t} + \bfq_h(t)\cdot\grad\varphi(t,x) }}+ \mathcal{O}(h),  \\
&\mbox{ with } \ \eta(\vU_h) =\vr_h\chi\left( \log \left(
\frac{(\gamma-1)}{\vr_h^{\gamma}}\left(E_h-\frac{1}{2}\frac{|\vm_h|^2}{\vr_h}\right)
\right)\right)\\
\mbox { for any }  \varphi \in C^3_c([0,T)\times\Omega), \, \varphi \geq 0, &\mbox{ and any } \chi \mbox{ defined on R, increasing, concave, } \chi(Z) \leq \Ov{\chi} \mbox{ for all } Z.
\end{aligned}
\end{equation}

\begin{Lemma}\label{L:AcE}
Let us assume that
\begin{itemize}
\item[(A1)] no vacuum appears, i.e., there exists $\Un{\vr} >0 $, such that $ \vr_h(t) \geq \Un{\vr},$ cf.~(\ref{pHYP})
\item[(A2)] entropy Hessian is strictly {\cg negative definite, i.e., there exists $\Un{\eta} > 0$, such that $\displaystyle \frac{d^2\eta(\vU)}{d\vU^2}  \leq - \Un{\eta}\mathbb I$}
\item[(A3)] numerical entropy flux $\bfQ_h$ is globally Lipschitz continuous, cf. \eqref{globLC}.
\end{itemize}
Then the local Lax-Friedrichs scheme (\ref{scheme_cE1}) is consistent with the complete Euler system (\ref{F5}) and the consistency
formulation (\ref{consistency_cE}) holds. If we assume that
\begin{itemize}
{\cg \item[(A1)] no vacuum appears, i.e., there exists $\Un{\vr} >0 $, such that $ \vr_h(t) \geq \Un{\vr}$, cf.~(\ref{pHYP})}
\item[(A2)] entropy Hessian is strictly {\cg negative definite, i.e., there exists $\Un{\eta}>0$, such that $\displaystyle \frac{d^2\eta(\vU)}{d\vU^2}  \leq - \Un{\eta}\mathbb I$}
\item[(A3)] numerical entropy flux $\bfQ_h$ is globally Lipschitz continuous, cf. \eqref{globLC}
\item[(A4)] finite speed of propagation holds, i.e., there  exists $\Ov{\lambda} >0$ such that $\lambda(\vU_h(t)) \leq \Ov{\lambda}$ uniformly for $t \in [0,T]$ and $h\rightarrow 0,$
\end{itemize}
then the global Lax-Friedrichs scheme (\ref{scheme_cE}) is consistent with the complete Euler system (\ref{F5}) and the consistency
formulation (\ref{consistency_cE}) holds.
\end{Lemma}


Recalling Lemmas~\ref{L}, \ref{L:AbE} and \ref{L:AcE}  we derive the following results.

\begin{Corollary}\label{Cor:bE}
Let $\vU_h=[\vr_h,\vm_h]$ be a numerical solution of the barotropic Euler system constructed by the global Lax-Friedrichs scheme \eqref{scheme_bE}. Suppose that there exist positive constants $\Un{\vr},$ $\Ov{\vr},$ $\Ov{\vm}>0$ such that
\begin{align*}
0 < \Un{\vr} \leq \vr_h \leq \Ov{\vr}, \ |\vm_h| \leq \Ov{\vm}, \mbox{ uniformly for} \ h  \rightarrow
 0.
\end{align*}
Then the assumptions (A1), (A3) of Lemma~\ref{L:AbE} are satisfied.

\medskip
\noindent Let $\vU_h=[\vr_h,\vm_h]$ be a numerical solution of the local Lax-Friedrichs scheme \eqref{scheme_bE1}.
For $\gamma \geq 3$ we suppose that  there exists constant $\Un{\vr},$  such that
\begin{align*}
0 < \Un{\vr} \leq \vr_h  \  \mbox{ uniformly for} \ h  \rightarrow 0,
\end{align*}
for $1 < \gamma < 3$ we suppose that there exist constants $\Un{\vr},$ $\Ov{\vr},$  such that
\begin{align*}
0 < \Un{\vr} \leq \vr_h \leq \Ov{\vr}, \  \mbox{ uniformly for} \ h  \rightarrow 0.
\end{align*}
Then the assumptions (A1), (A2) of Lemma~\ref{L:AbE} are satisfied.
Consequently, the global and the local Lax-Friedrichs schemes for the barotropic Euler equations satisfy the consistency formulation \eqref{consistency_bE}.
\end{Corollary}

\begin{Corollary}\label{Cor:cE}
 Let $\vU_h=[\vr_h,\vm_h,E_h]$ be a numerical solution of the complete Euler system constructed by the schemes \eqref{scheme_cE1} or \eqref{scheme_cE}. Suppose that there exist constants $\Un{\vr},$ $\Ov{E}>0$ such that
\begin{align*}
\Un{\vr} \leq \vr_h, \ E_h \leq \Ov{E}, \mbox{ uniformly for} \ h  \rightarrow
 0.
\end{align*}
Then the assumptions (A1)--(A5) of Lemma~\ref{L:AcE} are satisfied. In particular, the global and the local Lax-Friedrichs schemes for the complete Euler equations satisfy the consistency formulation \eqref{consistency_cE}.

\end{Corollary}

\section{Limit process}

Recall that for simplicity we assume that $\Omega= \left( [0,1]|_{\{ 0, 1\}} \right)^N$, $N=1,2,3$ is the flat torus, meaning we focus on spatially periodic solutions.
In addition, we prescribe regular initial data,
\begin{equation} \label{RiD}
\vr^0 \in C^1 (\Omega),\  \vr^0 > 0,\ \vc{m}^0 = C^1(\Omega; R^N),\ E^0 \in C^1(\Omega), \ p^0 = (\gamma - 1) \left( E^0 - \frac{1}{2} \frac{|\vc{m}^0|^2}{\vr^0} \right)
> 0.
\end{equation}
Under the perfect gas state equation, the last condition gives rise to the initial temperature,
\[
\vt^0 =  \frac{(\gamma - 1)}{\vr^0} \left( E^0 - \frac{1}{2} \frac{|\vc{m}^0|^2}{\vr^0} \right).
\]

\subsection{Generating measure--valued solutions}

\subsubsection{Equation of continuity, weak limit}

Let $\vr_h$, $\vc{m}_h$, and $E_h$ be a family of numerical solutions corresponding to the time step $h$.
The energy estimates \eqref{apriori_bE} and \eqref{apriori_cE} can be used to
deduce,  at least for suitable subsequences,
\[
\begin{split}
\vr_h &\to \vr \ \mbox{weakly-(*) in}\ L^\infty(0,T; L^{\gamma}(\Omega)),\ \vr \geq 0 \\
\vm_h  &\to \vm \ \  \mbox{weakly-(*) in}\ L^\infty(0,T; L^{r}(\Omega; R^N)), \ r=\frac{2\gamma}{\gamma+1}>1,
\end{split}
\]
for both the barotropic and the complete Euler systems.
In addition, it may be deduced from \eqref{consistency_bE} or \eqref{consistency_cE} that the limit functions satisfy the equation of continuity in the form
\begin{equation} \label{M10}
- \intO{ \vr^0 \varphi (0, \cdot) } = \int_0^T \intO{ \left[  \vr \partial_t \varphi + \vm \cdot \Grad \varphi \right] } \dt
\end{equation}
for any test function $\varphi \in C^1_c([0, T) \times \Omega )$. Clearly,
\[
\vr \in C_{\rm weak}([0,T]; \Omega)
\]
and (\ref{M10}) can be rewritten in the form
\begin{equation} \label{M11}
\left[ \intO{ \vr \varphi (t, \cdot) } \right]_{t = 0}^{t=\tau} = \int_0^\tau
\intO{ \left[  \vr \partial_t \varphi + \vm \cdot \Grad \varphi \right] } \dt
\end{equation}
for any $0 \leq \tau \leq T$ and any $\varphi \in C^1([0,T] \times \Omega)$.

\subsubsection{Young measure generated by numerical solutions}
\label{young}

The entropy inequality \eqref{dis_en_ineq}, along with the consistency formulations \eqref{consistency_bE} and \eqref{consistency_cE} provide a suitable platform for the use of the theory of dissipative
measure--valued solutions developed in \cite{FGSWW1}.
Consider the family of numerical solutions $[\vr_h, \vm_h, E_h]$ (complete Euler) or $[\vr_h, \vm_h]$ (barotropic Euler).
In accordance with the weak convergence statement
derived in the preceding part and boundedness of the total energy established in \eqref{coonservation_cE}, these families generate a Young measure - a parametrized measure
\[
\mathcal{V}_{t,x} \in L^\infty((0,T) \times \Omega; \mathcal{P}(\mathcal{F})) \ \mbox{for a.a.}\ (t,x) \in (0,T) \times \Omega,
\]
sitting on the phase space $\mathcal{F}$, where the latter is
\[
\mathcal{F} = \left\{ [\vr, \vc{m}] \in [0, \infty) \times R^N \right\}
\]
for the barotropic Euler system, and
\[
\mathcal{F} = \left\{ [\vr, \vc{m}, E] \ \Big|\  [0, \infty) \times R^N \times [0, \infty) \right\}
\]
for the complete Euler system. Recall that, in accordance with the fundamental theorem of the theory of Young measures (see e.g. Ball \cite{BALL2} or Pedregal \cite{PED1}),
we have
\[
\left< \mathcal{V}_{t,x}, g(\vc{U}) \right> = \Ov{g(\vc{U})}(t,x)\ \mbox{for a.a.}\ (t,x) \in (0,T) \times \Omega,
\]
whenever $g \in C_c(\mathcal{F})$, and
\[
g(\vc{U}_h) \to \Ov{g(\vc{U})} \ \mbox{weakly in}\ L^1((0,T) \times \Omega).
\]

\subsubsection{Continuity equation}

Accordingly, the equation of continuity \eqref{M11} can be written as
\begin{equation} \label{M12}
\left[ \intO{ \left< \mathcal{V}_{t, x} ; \vr \right> \varphi (t, \cdot) } \right]_{t= 0}^{t=\tau } = \int_0^\tau
\intO{ \left[ \left< \mathcal{V}_{t, x} ; \vr \right>   \partial_t \varphi +  \left< \mathcal{V}_{t, x} ; \vc{m} \right> \cdot \Grad \varphi \right] } \dt
\end{equation}
Note that there is no concentration measure in \eqref{M12}, i.e., $\mu_C^1=0.$

\subsubsection{Momentum equation}

We apply a similar treatment to the momentum equation \eqref{mom_bE} and \eqref{mom_cE}.
Using a priori bounds \eqref{apriori_bE} and \eqref{apriori_cE} we obtain that
\[
\frac{\vc{m}_h \otimes \vc{m}_h}{\vr_h} \ \mbox{is bounded in}\ L^\infty(0,T; L^1(\Omega; R^{N \times N})),
\]
and
\[
p_h \mbox{ is bounded in}\ L^\infty(0,T; L^1(\Omega)).
\]
Recall that the pressure is defined as
\[
p_h = \left\{ \begin{array}{l} a \vr_h^\gamma \ \mbox{in the barotropic case,}\\ \\
(\gamma - 1) \left( E_h - \frac{1}{2} \frac{|\vc{m}_h|^2}{\vr_h} \right) \ \mbox{for the complete system.} \end{array} \right.
\]
Thus, passing to subsequences as the case may be, we deduce
\[
\frac{\vc{m}_h \otimes \vc{m}_h}{\vr_h} + p_h \mathbb{I}  \to \Ov{ \frac{\vc{m}_h \otimes \vc{m}_h}{\vr_h} + p_h \mathbb{I} }
\ \mbox{weakly-(*) in}\ L^\infty(0,T; \mathcal{M}(\Omega; R^{N \times N} ) ).
\]
We set
\[
\mu^2_C \defeq \Ov{ \frac{\vm \otimes \vm}{\vr} + p \mathbb{I} } - \left< \mathcal{V}_{t,x}; \frac{\vm \otimes \vm}{\vr} + p \mathbb{I} \right>
\in L^\infty(0,T; \mathcal{M}(\Omega; R^{N \times N}))
\]
- the concentration measure appearing in the limit momentum equation.

Letting $h \to 0$ in \eqref{mom_bE} and \eqref{mom_cE} we conclude
\begin{equation} \label{M14}
\begin{aligned}
\left[ \intO{ \left<\mathcal{V}_{t,x} ; \vm \right> \cdot  \boldsymbol{\varphi}(0, \cdot) } \right]_{t = 0}^{t = \tau} &= \int_0^\tau \intO{
\Big[ \left< \mathcal{V}_{t,x}; \vm \right> \cdot \partial_t \boldsymbol{\varphi} + \left< \mathcal{V}_{t,x}; \frac{\vm \otimes \vm}{\vr} \right>
: \Grad \boldsymbol{\varphi} + \left< \mathcal{V}_{t,x}, p \right> \Div
\boldsymbol{\varphi}  \Big] } \ \dt \\
&
+ \int_0^\tau \intO{ \mu^2_C : \Grad \boldsymbol{\varphi}  } \dt
\end{aligned}
\end{equation}
for any $0 \leq \tau \leq T$, $\boldsymbol{\varphi} \in C^1([0,T] \times \Omega;R^N)$.

\subsubsection{Energy inequality for the barotropic Euler system}

In the barotropic case the energy plays the role of the entropy, cf. \eqref{entropy_bE}. A priori estimates \eqref{apriori_bE} indicate that the energy
\begin{align*}
\eta(\vU_h) = \frac{|\vm_h|^2}{2\vr_h} + P(\vr_h)
\end{align*}
is uniformly bounded in $L^{\infty}(0,T;L^1(\Omega)).$ Letting $h \to 0$ in \eqref{dis_en_ineq} for the barotropic Euler system we obtain
\begin{align*}
\left[ \intO{ \left<\mathcal{V}_{t,x} ; \eta(\vU_h(t)) \right>   } \right]_{t = 0}^{t = \tau} + \mathcal{D}(t)\leq 0,
\end{align*}
with the dissipation defect $\mathcal{D} \in L^{\infty}(0,T),$ $\mathcal{D}(t)\geq 0$, see \cite{FGSWW1} for details. Moreover,  applying~\cite[Lemma~2.1.]{FGSWW1} for
 $$F(\vU_h(t))=\intO{\frac{\vm_h(t)\otimes\vm_h(t)}{\vr_h(t)}+p_h(t)\mathbb{I}}, \quad G(\vU_h(t))=\intO{\eta(\vU_h(t))},  \ \mbox{a.a.} \ t \in (0,T) ,$$
 we get the compatibility condition \eqref{M9}, specifically
\begin{align*}
 \int_{{\Omega}} 1 \ {\rm d} |\mu^2_C| \aleq \mathcal{D} \ \mbox{a.a. in}\ (0,T).
\end{align*}

\subsubsection{Entropy inequality and energy balance for the complete Euler system}

\subsubsection*{Entropy inequality}
Due to a priori estimates the entropy pair $(\eta(\vU_h),\bfq(\vU_h))$ for the complete Euler system, cf. \eqref{entropy_cE} and \eqref{entropy_flux_cE}, is uniformly bounded in
$[L^\infty(0,T; L^{\gamma}(\Omega))] \times [L^\infty(0,T; L^r(\Omega))]^N.$
Therefore we have
\begin{align*}
\eta(\vU_h) &\to \Ov{\eta(\vU)} \ \mbox{weakly-(*) in}\ L^\infty(0,T; L^{\gamma}(\Omega)), \\
\bfq(\vU_h) &\to \Ov{\bfq(\vU)} \ \mbox{weakly-(*) in}\ L^\infty(0,T; L^r(\Omega)), \ r=\frac{2\gamma}{\gamma+1}>1.
\end{align*}
Letting $h \to 0$ in the equation \eqref{consistency_cE}, we get analogously as before,
\begin{equation} \label{M15}
\begin{aligned}
\left[ \intO{ \left<\mathcal{V}_{t,x} ; \eta(\vU) \right> \cdot  \varphi(0, \cdot) } \right]_{t = 0}^{t = \tau} &\geq \int_0^\tau \intO{
\Big[ \left< \mathcal{V}_{t,x}; \eta(\vU)\right> \cdot \partial_t \varphi + \left< \mathcal{V}_{t,x}; \bfq(\vU)\right>
\cdot \grad \varphi \Big] } \ \dt
\end{aligned}
\end{equation}
for a.a. $0 \leq \tau \leq T$, and any $\varphi \in C^1([0,T] \times \Omega),$ $\varphi\geq 0.$

\subsubsection*{Energy balance}

Equation \eqref{en_cE} of the complete Euler system  yields 
the discrete energy balance
 \begin{align}\label{M16}
 \left[ \intO{E_h(t)}\right]_{t=0}^{t=\tau}=0.
 \end{align}
  Letting $h\to 0$ in \eqref{M16} and taking into account that $\{E_h\}_{h>0}$ is uniformly bounded in $L^\infty(0,T; L^1(\Omega))$ we obtain
 \begin{align*}
 \left[ \intO{ \left<\mathcal{V}_{t,x} ; E_h(t) \right>   } \right]_{t = 0}^{t = \tau} + \mathcal{D}(t)=0,
 \end{align*}
 where $\mathcal{D} \in L^{\infty}(0,T),$ $\mathcal{D} \geq 0.$ We again apply \cite[Lemma~2.1.]{FGSWW1} for
 $$F(\vU_h(t))=\intO{\frac{\vm_h(t)\otimes\vm_h(t)}{\vr_h(t)}+p_h(t)\mathbb{I}}, \quad G(\vU_h(t))=\intO{E_h(t)},  \ \mbox{a.a.} \ t \in (0,T), $$
to get that
 \begin{align*}
 \int_{{\Omega}} 1 \ {\rm d} |\mu^2_C| \aleq \mathcal{D} \ \mbox{a.a. in}\ (0,T).
\end{align*}

Summarizing the discussion of this section we are ready to formulate the following result.

\begin{Theorem}\label{T_gcE}\ \\
Let the initial data satisfy \eqref{RiD}. Let $\vU_h=[\vr_h,\vm_h,E_h]$ be a numerical solution of the complete Euler system constructed by the schemes \eqref{scheme_cE1} or \eqref{scheme_cE}. In addition, suppose that there exist constants $\Un{\vr},$ $\Ov{E}>0$ such that
\begin{align}\label{hyp_theo_generating_cE}
\Un{\vr} \leq \vr_h, \ E_h \leq \Ov{E}, \mbox{ uniformly for} \ h  \rightarrow
 0.
\end{align}
Then $\{\vU_h\}_{h>0}$ up to a subsequence generates a Young measure $$\mathcal{V}_{t,x} \in L^{\infty}_{weak(*)}((0,T)\times \Omega, \mathcal{P}([0,\infty)\times R^N \times [0,\infty)))$$ representing a  (DMV) solution of the complete Euler system in the sense of Definition~\ref{D2}.
\end{Theorem}

Note that hypothesis \eqref{hyp_theo_generating_cE} is considerably weaker than the standard stipulation
\begin{align}\label{standard_hyp_cE}
\|\vU_h\|_{L^{\infty}} \leq C, \ 0 < \Un{\vr} \leq \vr_h, \ 0 < \Un{E} \leq E_h,
\end{align}
cf. \cite{JR,FMT16,FB05,CL96,FKMT17}. The missing piece of information between \eqref{hyp_theo_generating_cE} and \eqref{standard_hyp_cE} is provided by the careful analysis of the renormalized entropy  inequality in Section~\ref{S:positive}, see Lemma~\ref{L}.

Similar result can be shown in the context of the barotropic Euler system.

\begin{Theorem}\label{T_gbE}\ \\
Let the initial data $\vr^0,$ $\vm^0$ be as in \eqref{RiD}. Let $\vU_h=[\vr_h,\vm_h]$ be a numerical solution of the barotropic Euler system constructed by the schemes \eqref{scheme_bE1} or \eqref{scheme_bE}. In addition,
\begin{itemize}
\item if $\vU_h$ is generated by the scheme \eqref{scheme_bE1}  and $\gamma \geq 3,$ we suppose
\begin{align}\label{hyp_theo_generating_bEa}
0 < \Un{\vr} \leq \vr_h  \mbox{ uniformly for} \ h  \rightarrow 0,
\end{align}
\item if $\vU_h$ is generated by the scheme \eqref{scheme_bE1}  and $1 < \gamma < 3,$ we suppose
\begin{align}\label{hyp_theo_generating_bEa}
0 < \Un{\vr} \leq \vr_h  \leq \Ov{\vr}  \mbox{ uniformly for} \ h  \rightarrow 0,
\end{align}
\item if $\vU_h$ is generated by the scheme \eqref{scheme_bE}, we suppose
\begin{align}\label{hyp_theo_generating_bE}
0 < \Un{\vr} \leq \vr_h \leq \Ov{\vr}, \ |\vm_h| \leq \Ov{\vm}, \mbox{ uniformly for} \ h  \rightarrow 0,
\end{align}
\end{itemize}
for certain positive constants $\Un{\vr},$ $\Ov{\vr}$ and $\Ov{\vm}$.\\
Then $\{\vU_h\}_{h>0}$ up to a subsequence generates a Young measure $$\mathcal{V}_{t,x} \in L^{\infty}_{weak(*)}((0,T)\times \Omega, \mathcal{P}([0,\infty)\times R^N ))$$ representing a (DMV) solution of the barotropic Euler system in the sense of Definition~\ref{D1}.
\end{Theorem}

It should be pointed out that for the barotropic Euler system the only available mathematical entropy is the energy, and in addition, its flux can  not be controlled in the asymptotic limit for $h \rightarrow 0$ unless we assume \eqref{hyp_theo_generating_bE}.

\subsection{Convergence to regular solution}

We have proven that the numerical solutions $\{\vU_h\}_{h>0}$  to \eqref{scheme_bE} and \eqref{scheme_cE} for the  barotropic and the complete Euler system converges to the dissipative measure--valued solution defined in Definition~\ref{D1} and Definition~\ref{D2}, respectively.
Employing the corresponding (DMV)-strong uniqueness results from \cite{GSWW} and \cite{BF} we can show the strong convergence to the strong solution of the system on its lifespan.

\begin{Theorem}\label{T_ccE}\ \\
Suppose that the
approximate solutions $\{\vU_h\}_{h>0}$  to  \eqref{scheme_cE1} or \eqref{scheme_cE} for the complete Euler system generate
a (DMV) solution in the sense of Definition~\ref{D2}.
In addition, let the Euler equations  \eqref{F5} possess the unique strong (continuously differentiable) solution $\vU=[\vr,\vm,E]$,
emanating form the initial data \eqref{RiD}. \\
Then
\begin{align*}
\vU_h \to \vU \ \mbox{strongly in} \ L^1((0,T) \times \Omega;  \mathcal{P}([0,\infty)\times R^N \times [0,\infty))).
\end{align*}
More precisely,
\begin{equation}\label{PP1}
 \begin{aligned}
\vr_h &\to \vr \ \mbox{weakly-(*) in} \ L^{\infty}(0,T;L^\gamma(\Omega))
\ \mbox{and strongly in}\ L^1((0,T) \times \Omega) \\
\vm_h &\to \vm \ \mbox{weakly-(*) in} \ L^{\infty}(0,T;L^{2\gamma/(\gamma-1)}(\Omega)) \ \mbox{and strongly in}\ L^1((0,T) \times \Omega; R^N)), \\
E_h &\to E \ \mbox{weakly-(*) in} \ L^{\infty}(0,T;L^1(\Omega)) \ \mbox{and strongly in}\ L^1((0,T) \times \Omega).
\end{aligned}
\end{equation}
\end{Theorem}

\begin{Remark}
Recall that the strong solution of the complete Euler system conserves energy, in particular, the dissipation defect $\mathcal{D}$, and, accordingly, the concentration
measure $\mu^2_C$ vanish. This also justifies the strong convergence of the total energy claimed in (\ref{PP1}).
\end{Remark}

In contrast with Theorem~\ref{T_gcE} the results stated in Theorem~\ref{T_ccE} is unconditional provided that:
\begin{itemize}
\item the limit system admits a smooth solution.
\item  the numerical solution generates a (DMV) solution.
\end{itemize}

Exactly the same result can be obtained for the barotropic Euler system \eqref{M1}  and the entropy stable finite volume schemes \eqref{scheme_bE1} and \eqref{scheme_bE}.

\section*{Conclusions}

 We have shown  convergence of the Lax-Friedrichs-type finite volume schemes for multidimensional barotropic and complete Euler equations.
Since multidimensional Euler equations are ill-posed in the class of weak solutions for $L^\infty$-initial data \cite{FKKM17}, we propose here to investigate the convergence in the class of  dissipative measure--valued (DMV) solutions. The latter  has been introduced for the Euler equations recently in \cite{BF, BreFei17B, FKKM17}, see also the related works on the (DMV) solutions of the compressible Navier-Stokes equations \cite{FL17, FGSWW1}. The (DMV) solutions represent the most general class of solutions that still satisfy the weak--strong uniqueness property. Thus, if the strong solution exists the (DMV) solution coincides with the strong one on its lifespan, cf.~\cite{GSWW} and \cite{BF} for the barotropic and complete Euler equations, respectively.

We build on the concept of entropy stable schemes that has been introduced  by Tadmor \cite{Tad87}, see also \cite{Tad03} and the references therein. We work here with the Lax-Friedrichs-type finite volume schemes (\ref{scheme_bE1}), (\ref{scheme_bE}) and (\ref{scheme_cE1}), (\ref{scheme_cE})
that are  entropy stable.  Furthermore, using some refined a priori estimates
for the numerical solutions we have shown  consistency of our entropy stable schemes. More precisely, assuming only strict positivity of the density and the upper bound on the energy we have proven the consistency for the complete Euler system, cf.~Corollary~\ref{Cor:cE}.
On the other hand, the consistency of the local Lax-Friedrichs scheme
(\ref{scheme_bE1}) for barotropic Euler equations with $\gamma \geq 3$ can be obtained assuming only the strict positivity of density, cf.~Lemma~\ref{L:AbE}.
In Theorems~\ref{T_gcE}, \ref{T_gbE} we have shown that numerical solutions given by the Lax-Friedrichs-type finite volume schemes generate the Young measure representing (DMV) solutions of the complete and barotropic Euler equations, respectively. Employing the corresponding (DMV)-- strong uniqueness results we have shown in Theorem~\ref{T_ccE} the strong convergence to the strong solution of the complete Euler system on its lifespan. Analogous strong convergence result holds for the barotropic Euler equations, too.


\end{document}